\documentclass{article}

 \usepackage[preprint]{neurips_2019}

\usepackage[utf8]{inputenc} 
\usepackage[T1]{fontenc}    
\usepackage{url}            
\usepackage{booktabs}       
\usepackage{amsfonts}       
\usepackage{nicefrac}       
\usepackage{microtype}      
\usepackage{amsthm}
\usepackage{multicol,multirow,float,tabu}
\usepackage{amsmath}
\usepackage{algorithm,algorithmic}
\usepackage{tikz,pgfplots}
\usetikzlibrary{matrix}
\usepgfplotslibrary{groupplots}

\def\RR{\mathbb{R}}
\def\EE{\mathbb{E}}
\def\ZZ{\mathbb{Z}}
\def\DD{\Delta}

\newcommand{\prox}{\operatorname{prox}}

\newcommand{\argmin}{\operatornamewithlimits{argmin}}
\DeclareMathOperator*{\dist}{dist}
\def\tPhi{{\tilde{\Phi}}}
\def\G{{\cal G}}

\def\P{{\cal P}}
\def\S{{\cal S}}

\newtheorem{theorem}{Theorem}
\newtheorem{lemma}[theorem]{Lemma}
\newtheorem{corollary}[theorem]{Corollary}
\newtheorem{property}[theorem]{Property}

\title{Stochastic Proximal Methods for Non-Smooth Non-Convex Constrained Sparse Optimization}

\author{%
  Michael R. Metel \\
  RIKEN Center for Advanced Intelligence Project \\
  Tokyo, Japan \\
  \texttt{michaelros.metel@riken.jp} 
  \AND
  Akiko Takeda \\
  Department of Creative Informatics\\
  Graduate School of Information Science and Technology \\
  University of Tokyo\\
  RIKEN Center for Advanced Intelligence Project\\
  Tokyo, Japan\\
  \texttt{takeda@mist.i.u-tokyo.ac.jp}
}

\begin{document}

\maketitle

\begin{abstract}
This paper focuses on stochastic proximal gradient methods for optimizing a smooth non-convex loss 
function with a non-smooth non-convex regularizer and convex constraints. To the best of our knowledge 
we present the first non-asymptotic convergence results for this class of problem. We present two 
simple stochastic proximal gradient algorithms, for general stochastic and finite-sum optimization 
problems, which have the same or superior convergence complexities compared to the current best 
results for the unconstrained problem setting. In a numerical experiment we compare our algorithms 
with the current state-of-the-art deterministic algorithm and find our algorithms to exhibit superior 
convergence.
\end{abstract}

\section{Introduction}
\label{Int}

In this paper we consider optimization problems of the form
\begin{alignat}{6}\label{eq:1}
&\min\limits_{w\in\RR^d}&&\text{ }\Phi(w):=f(w)+g(w)+h(w),
\end{alignat}
where $f(w)$ has a Lipschitz continuous gradient and $h(w)$ is a proper closed 
convex function. We assume that $g(w)$ and $h(w)$ have proximal operators that can be efficiently 
computed. In addition, we assume that 
\begin{alignat}{6}\label{eq:genobj}
&f(w):=\EE_{\xi}[F(w,\xi)],
\end{alignat}
where $\xi\in \RR^p$ is a random vector following a probability distribution $P$ from which 
i.i.d. samples can be generated. We will also consider the finite-sum problem with 
\begin{alignat}{6}\label{eq:fsobj}
&f(w):=\frac{1}{n}\sum_{j=1}^n f_j(w),
\end{alignat}
where each $f_j(w)=F(w,\xi_j)$ has a Lipschitz continuous gradient.

The function $f(w)$ is intended to model the objective of our optimization problem, such as a loss 
function in empirical risk minimization, an agent's utility function in portfolio optimization, 
or a statistical procedure we want to perform on collected data, where non-convex smooth functions 
arise naturally. In particular, non-convex loss functions have been shown to 
achieve better generalization \citep{shen2003}, prospect theory \citep{kahneman1979} motivates the use 
of S-shaped utility functions, and principal components of a dataset can be computed by using a 
non-convex smooth objective function. 

In many applications of optimization, a sparse 
solution is desirable as it avoids overfitting to sampled data, and simplifies the interpretation of 
the result and its implementation. Our motivation for $g(w)$ is to be a non-smooth 
non-convex regularizer, such as SCAD \citep{fan2001}, MCP \citep{zhang2010}, the log-sum penalty 
\citep{candes2008}, or the capped $l_1$ norm, which are able to better approximate the $l_0$ norm than 
their convex or smooth counterparts. The function $h(w)$ allows us to include convex constraints 
to our problem through the use of an indicator function of the convex feasible region. 

We now present one concrete example of a sparse constrained optimization problem which fits within our 
assumptions, and which will also be used in our numerical experiments. Two more involved  
applications, {\em 
Sparse binary classification with outlier detection and fairness 
constraints} and {\em Sparse portfolio optimization using S-shaped utility with loss aversion} are 
included in Section 1 of the supplementary material, with implementation details for all three 
applications.

{\em Non-negative sparse principal component analysis:} Principal component analysis (PCA) finds a 
lower dimensional approximation of a dataset, with the non-negative sparse extension having 
applications in economics, bioinformatics and computer vision \citep{zass2007}. Given a data 
set $x\in \RR^{d\times n}$, we find its first sparse non-negative principal component by solving
\begin{alignat}{6}\label{eq:ap1}
&\min&&\text{ }-\frac{1}{2n}\sum_{j=1}^n(w^Tx_j)^2+g(w)\\
&\text{s.t.}&&||w||_2\leq 1, w\geq 0.\nonumber
\end{alignat}
The objective $f(w)=-\frac{1}{2n}\sum_{j=1}^n(w^Tx_j)^2$ is a smooth non-convex function, $g(w)$ can 
be taken as one of the non-smooth non-convex regularizers previously mentioned, and the constraints 
have a closed form projection \citep[Theorem 7.1]{bauschke2018}.

{\bf Related work:} First order stochastic methods for the case of a non-smooth convex regularizer 
$g(w)$ with $h(w)=0$ is an active research area. Non-asymptotic convergence results 
were first achieved in \citep{ghadimi2016}. For finite-sum problems, \citet{reddi2016pro} were the 
first to develop a proximal stochastic variance reduced gradient algorithm with improved convergence 
complexity. 

For the problem of solving \eqref{eq:1} where neither the function $f(w)$ nor $g(w)$ is 
convex, the current body of research is limited. \citet{kawashima2018} consider $g(w)$ as a 
non-smooth quasi-convex function and achieve the same convergence complexity as in 
\citep{ghadimi2016}. The only 
other non-asymptotic 
convergence results for a non-smooth non-convex function $g(w)$ to our knowledge are found in 
\citep{xu2018} and \citep{metel2019}. \citet{xu2018} assume that $f(w)=f^1(w)-f^2(w)$, where 
both $f^1(w)$ and $f^2(w)$ are convex, $f^1(w)$ is smooth, $f^2(w)$ has a H\"{o}lder continuous 
gradient, and $h(w)=0$. In \citep{metel2019} it is assumed that $h(w)=0$. 

{\bf Our contributions:}
\begin{itemize}
	\item We present a mini-batch stochastic proximal algorithm for general 
	stochastic objectives of the form \eqref{eq:genobj}, and a variance reduced stochastic 
	proximal algorithm for finite-sum problems of the form \eqref{eq:fsobj}. We are not aware of any 
	other works proving non-asymptotic convergence for this type of problem. 
	
	\item We achieve the same or better convergence complexities as demonstrated in 
	\citep{xu2018,metel2019} while considering a more general problem setting, which are summarized 
	in Table \ref{t:1}. The complexities are in terms 
	of the number of gradient calls and proximal operations, see Section \ref{prelims}. 
	
	\item We implement both algorithms and show superior convergence compared 
	to a state-of-the-art deterministic algorithm. 
\end{itemize} 
\begin{table*}[ht]
	\caption{Comparison of convergence complexities obtained in \citep{xu2018,metel2019} (with 
	$h(w)=0$) and this paper.}
	\label{t:1}
	\begin{center}
		\begin{small}
			\begin{sc}
				\tabulinesep=0.6mm
				\tabcolsep=1.8mm
				\begin{tabu}{llccc}
					\hline
					Algorithm	& Reference &\begin{tabular}{@{}c@{}}Finite\\-sum\end{tabular} & 
					\begin{tabular}{@{}c@{}}Gradient Call\\ 
						Complexity\end{tabular} & \begin{tabular}{@{}c@{}}Proximal Operator\\ 
						Complexity\end{tabular}\\
					\hline
					MBSPA & Corollary \ref{mbcomp}& 
					$\times$ & $O(\epsilon^{-5})$ &$O(\epsilon^{-3})$\\
					VRSPA & Corollary \ref{fscomp}& 
					$\surd$ & $O(n^{2/3}\epsilon^{-3})$ 
					&$O(\epsilon^{-3})$\\
					MBSGA & \begin{tabular}{@{}l@{}}Corollary 7 ,\\\citep{metel2019}\end{tabular}& 
					$\times$ & $O(\epsilon^{-5})$ &$O(\epsilon^{-4})$\\
					VRSGA & \begin{tabular}{@{}l@{}}Corollary 12 ,\\\citep{metel2019}\end{tabular}& 
					$\surd$ & $O(n^{2/3}\epsilon^{-3})$ 
					&$O(\epsilon^{-3})$\\				
					SSDC-SPG & \begin{tabular}{@{}l@{}}Theorem 7 a,\\\citep{xu2018}\end{tabular}
					& $\times$ & $O(\epsilon^{-5})$ & $O(\epsilon^{-5})$ \\
					SSDC-SVRG & \begin{tabular}{@{}l@{}}Theorem 7 c,\\\citep{xu2018}\end{tabular} & 
					$\surd$  & $\tilde{O}(n\epsilon^{-3})$ & 
					$\tilde{O}(\epsilon^{-3})$\\
					\hline
				\end{tabu}
			\end{sc}		
		\end{small}
	\end{center}
	\vskip -0.1in
\end{table*}

\section{Background}
\label{prelims}
We assume that $f(w)$ has a Lipschitz continuous gradient with parameter $L$, 
$$||\nabla f(w)-\nabla f(x)||_2\leq L||w-x||_2,$$
which we will denote as being an $L$-smooth function. In the finite-sum case, we assume that each 
$f_j(w)$ is $L$-smooth. Given a sample $\xi^k\sim P$, generated in iteration $k$ of an algorithm, we 
assume we can generate an unbiased stochastic gradient $\nabla F(w,\xi^k)$ such that  
\begin{alignat}{6}
&\EE[\nabla F(w,\xi^k)]=\nabla f(w),\label{eq:unbgrad}
\end{alignat}
and for some constant $\sigma$,
\begin{alignat}{6}
&\EE||\nabla F(w,\xi^k)-\nabla f(w)||^2_2\leq \sigma^2.\label{vbound}
\end{alignat}
Let $\partial \Phi(w)$ denote the limiting subdifferential of our objective, defined as
$$\partial \Phi(w):=\{v: \exists w^k \xrightarrow{\Phi} w, v^k\in \hat{\partial}\Phi(w^k)\text{ with } 
v^k\rightarrow v\},$$
where $\hat{\partial}\Phi(w):=\{v: \liminf\limits_{x\rightarrow w,x\neq w} 
\frac{\Phi(x)-\Phi(w)-\langle 
	v,x-w\rangle}{||x-w||_2}\geq 0\}$ and $w^t\xrightarrow{\Phi}w$ signifies the sequence 
$w^k\rightarrow 
w$ and $\Phi(w^k)\rightarrow \Phi(w)$. The limiting subdifferential is equal to the gradient and 
subdifferential when the function is continuously differentiable and proper convex, respectively. We 
also assume the 
proximal operators of $g(w)$ and $h(w)$ are nonempty for all $w$, and that they can be computed 
efficiently, 
$$\prox_{\lambda g}(w):=\argmin\limits_{x\in 
	\RR^d}\left\{\frac{1}{2\lambda}||w-x||^2_2+g(x)\right\}$$ 
\begin{alignat}{6}
\prox_{\gamma h}(w):=\argmin\limits_{x\in 
	\RR^d}\left\{\frac{1}{2\gamma}||w-x||^2_2+h(x)\right\},\label{proxh}
\end{alignat}
for $\lambda,\gamma>0$. In particular, let us denote an element of $\prox_{\lambda g}(w)$ as  
\begin{alignat}{6}
&\zeta^{\lambda}(w)\in \prox_{\lambda g}(w).\label{zprox}
\end{alignat}
We note that $\prox_{\gamma h}(w)$ maps to a singleton since $h(w)$ is proper, closed, and convex, see 
for example \citep[Theorem 6.3]{beck2017}. 

We are interested in the convergence complexity of finding 
an $\epsilon$-accurate solution, using what we call the {\it subdifferential mapping},
$$\P_{\gamma}(w,\S):=\left\{\frac{1}{\gamma}\left(w-\prox_{\gamma h}(w-\gamma 
s)\right):s\in\S\right\}\label{gsgm},$$
where $\S\subseteq\RR^d$ is the subdifferential of a function, which is a closed set wherever the 
function is finite \citep[Theorem 8.6]{rockafellar2009}. In particular, for 
$$\G_{\gamma}(w):=\P_{\gamma}(w,\nabla f(w)+\partial g(w)),$$
we are interested in algorithm solutions $\bar{w}$, with accompanying $\bar{\gamma}>0$, which 
satisfy   
\begin{alignat}{6}
&\EE\left[\dist(0,\G_{\bar{\gamma}}(\bar{w}))\right]\leq \epsilon.\label{epsol}
\end{alignat}
We will also use the notation 
$\P_{\gamma}(w,G)$, where $G\in\RR^d$ is the gradient or a particular subgradient of a function in our 
analysis. $\G_{\gamma}(w)$ generalizes the gradient mapping $\P_{\gamma}(w,\nabla f(w))$ 
which has been used in the convergence criterion for  
proximal stochastic gradient methods for solving \eqref{eq:1} with $g(w)=0$, such as in 
\citep{ghadimi2016,reddi2016pro,li2018}. To motivate our measure of convergence \eqref{epsol}, 
consider the case where 
$$\dist(0,\G_{\bar{\gamma}}(\bar{w}))=0.$$ 
This implies that there exists an element 
$s_g(\bar{w})\in \partial g(\bar{w})$ such that 
\begin{alignat}{6}
&0=\P_{\bar{\gamma}}(\bar{w},\nabla 
f(\bar{w})+s_g(\bar{w})),\label{zeroP}
\end{alignat} 
and in particular 
\begin{alignat}{6}
&\bar{w}=\prox_{\bar{\gamma} h}(\bar{w}-\bar{\gamma}(\nabla f(\bar{w})+s_g(\bar{w}))).\label{zeroh}
\end{alignat}
From the first order optimality condition of $\prox_{\bar{\gamma} h}(\bar{w}-\bar{\gamma}(\nabla 
f(\bar{w})+s_g(\bar{w})))$ in \eqref{proxh},
$$0\in -\P_{\bar{\gamma}}(\bar{w},\nabla 
f(\bar{w})+s_g(\bar{w}))+\nabla f(\bar{w})+s_g(\bar{w})+\partial h(\prox_{\bar{\gamma} 
	h}(\bar{w}-\bar{\gamma} 
(\nabla f(\bar{w})+s_g(\bar{w})))).$$
Applying \eqref{zeroP} and \eqref{zeroh},
$$0\in \nabla f(\bar{w})+\partial g(\bar{w})+\partial h(\bar{w}).$$
We will measure algorithm complexity in terms of the number of gradient calls and proximal operations. 
A gradient call is either computing $\nabla F(w,\xi^k)$ given a sample $\xi^k$, or in the 
finite-sum case, returning $\nabla f_j(w)$ for a given $j$.

\section{Auxiliary functions of $\Phi(w)$}
This section uses the technique found in \citep{metel2019} with the addition of a function $h(w)$. Our 
algorithms rely on a sequence of majorizing functions of 
$$\tPhi_{\lambda}(w):=f(w) + e_{\lambda}g(w)+h(w)$$
where $g(w)$ has been replaced by its Moreau envelope,
$$e_{\lambda}g(w):=\inf_{x\in \RR^d} \left\{\frac{1}{2\lambda}||w-x||^2_2+g(x)\right\},$$
in $\Phi(w)$. Taking $x=w$, we note that 
\begin{alignat}{6}
&&&e_{\lambda}g(w)\leq g(w).\label{melb}
\end{alignat}
Given iteration $w^k$, a smooth majorizing function of $f(w) + e_{\lambda}g(w)$ can be 
written as 
\begin{alignat}{6}
&E_{\lambda}^k(w):=f(w)+U^k_{\lambda}(w),\label{Efunc}
\end{alignat}
where 
$$U^k_{\lambda}(w)=\frac{1}{2\lambda}||w||^2_2-\left(D^{\lambda}(w^k)+\frac{1}{\lambda}\zeta^{\lambda}(w^k)^T(w-w^k)\right),$$
and
\begin{alignat}{6}
&D^{\lambda}(w)=\sup_{x\in 
	\RR^d}\left(\frac{1}{\lambda}w^Tx-\frac{1}{2\lambda}||x||^2_2-g(x)\right).\label{bigD}
\end{alignat}

We will only need to evaluate the gradient of $E_{\lambda}^k(w)$, which is simply  
\begin{alignat}{6}
&\nabla E_{\lambda}^k(w)=\nabla f(w)+\frac{1}{\lambda}(w-\zeta^{\lambda}(w^k)).\label{grade}
\end{alignat}
\begin{property}
	\label{Eprop}	
	The following holds for $E^k_{\lambda}(w)$.		
	\begin{alignat}{6}
	&E_{\lambda}^k(w)+h(w)\geq \tPhi_{\lambda}(w)\hspace{0 mm} \text{ for all 
	}w\in\RR^d\nonumber&\\
	&E_{\lambda}^k(w^k)+h(w^k)=\tPhi_{\lambda}(w^k)\nonumber&\\
	&E_{\lambda}^k(w) \text{ is } 
	L_{\lambda}:=\left(L+\frac{1}{\lambda}\right)-\text{smooth.}&\nonumber
	\end{alignat}	
\end{property}
For completeness, we provide the proof of Property 1 in Section 2 of the supplementary material.

\section{Mini-batch stochastic proximal algorithm}
The algorithm presented in this section makes use of 
\begin{alignat}{6}
\label{Agrad}	
&\nabla A^k_{\lambda, M}(w,\xi^k)=\frac{1}{M}\sum_{j=1}^{M}\nabla 
F(w,\xi^k_j)+\frac{1}{\lambda}(w-\zeta^{\lambda}(w^k)),
\end{alignat}
which is a stochastic version of $\nabla E_{\lambda}^k(w)$, replacing $\nabla f(w)$ with an unbiased 
estimate using $M$ samples $\xi^k_j$, $j=1,...,M$ in iteration $k$.
\begin{algorithm}[H]
	\caption{Mini-batch stochastic proximal algorithm (MBSPA)}
	\label{alg:one}
	\begin{algorithmic}
		\STATE {\bfseries Input:} $w^{1}\in \RR^d$, $N\in\ZZ_{>0}$, $\alpha,\theta\in \RR$	
		\STATE $M:=\lceil N^{\alpha}\rceil\text{, }\lambda=\frac{1}{N^{\theta}}$
		\STATE $L_{\lambda}=L+\frac{1}{\lambda}\text{, }\gamma=\frac{1}{L_{\lambda}}$
		\STATE $R\sim \text{uniform}\{1,...,N\}$
		\FOR{$k=1,2,...,R-1$} 
		\STATE $\zeta^{\lambda}(w^k)\in \prox_{\lambda g}(w^k)$
		\STATE Sample $\xi^k\sim P^{M}$
		\STATE Compute $\nabla A^k_{\lambda, M}(w,\xi^k)$ \eqref{Agrad}
		\STATE $w^{k+1}=\prox_{\gamma h}(w^k-\gamma\nabla A^k_{\lambda, M}(w^k,\xi^k))$
		\ENDFOR
		\STATE {\bfseries Output:} $\bar{w}^R\in \prox_{\lambda g}(w^R)$
	\end{algorithmic}
\end{algorithm}

\subsection{Convergence analysis}
The convergence analysis of MBSPA follows the technique of \cite{ghadimi2016} 
adapted to our problem. We first define the following gradient mappings in iteration $k$,  
$$\G^k_{\gamma, A}(w^k):=\P_{\gamma}(w^k,\nabla A^k_{\lambda, M}(w^k,\xi^k))$$
and
$$\G^k_{\gamma, E}(w^k):=\P_{\gamma}(w^k,\nabla E^k_{\lambda}(w^k))\label{stgm}.$$
We also note that 
\begin{alignat}{6}
\label{gradmap}
&w^{k+1}&=&\prox_{\gamma h}(w^k-\gamma\nabla A^k_{\lambda, M}(w^k,\xi^k))\nonumber\\
&&=&w^k-\gamma\left(\frac{1}{\gamma}\left(w^k-\prox_{\gamma h}(w^k-\gamma 
\nabla A^k_{\lambda, M}(w^k,\xi^k))\right)\right)\nonumber\\
&&=&w^k-\gamma\G^k_{\gamma, A}(w^k).
\end{alignat}
The following lemma bounds $\EE\left[||\G^R_{\gamma, E}(w^R)||^2_2\right]$, which will be used to 
bound $\EE\left[\dist(0,\G_{\bar{\gamma}}(\bar{w}^R))\right]$ in Theorem \ref{mbconv}. 
The proof can be found in Section 3 of the supplementary material.
\begin{lemma}	
	\label{th:1}
	For an initial value $w_1\in \RR^d$, $N\in\ZZ_{>0}$, and $\alpha,\theta \in \RR$, MBSPA 
	generates 
	$w^R$ satisfying the following bound.  	
	\begin{alignat}{6}
	&&\EE||\G^R_{\gamma, E}(w^R)||^2_2&\leq
	\frac{(L+N^{\theta})}{N}\tilde{\DD}+\frac{6}{\lceil
		N^{\alpha}\rceil}\sigma^2,\nonumber
	\end{alignat}
	where $\tilde{\DD}=4(\tPhi_{\lambda}(w^1)-\tPhi_{\lambda}(w^*_{\lambda}))$ and 
	$w^*_{\lambda}$ is a global minimizer of $\tPhi_{\lambda}(\cdot)$.
\end{lemma}
In order to prove the convergence of $\EE\left[\dist(0,\G_{\bar{\gamma}}(\bar{w}^R)\right]$, we 
require the following two properties, the proofs of which can be found in Section 4 of the 
supplementary material. 
\begin{property}	
	\label{eq:switch}
	Assume that $g(w)$ is Lipschitz continuous with parameter $l$ and $\bar{\gamma}\geq \gamma$, then
	\begin{alignat}{6}
	&&\dist(0,\G_{\bar{\gamma}}(\zeta^{\lambda}(w^k)))&\leq 	
	&||\G^k_{\gamma, E}(w^k)||_2+2l\lambda\left(\frac{2}{\bar{\gamma}}+L\right).\nonumber
	\end{alignat}	
\end{property}
\begin{property}	
	\label{le:2}	
	Let $w^*$ be a global minimizer of $\Phi(\cdot)$ and let $w^*_{\lambda}$ be a global minimizer 
	of $\tPhi_{\lambda}(\cdot)$. Assume that $g(w)$ is Lipschitz continuous with 
	parameter $l$, then  
	$$\tPhi_{\lambda}(w)-\tPhi_{\lambda}(w^*_{\lambda})\leq 
	\Phi(w)-\Phi(w^*) + \frac{l^2\lambda}{2}.$$	
\end{property}
\begin{theorem}	
	\label{mbconv}	
	Assume that $g(w)$ is Lipschitz continuous with parameter $l$ and 
	$\bar{\gamma}=\frac{1}{N^{\tau}}$ for 
	$\tau\leq \theta$. The output $\bar{w}^R$ of 
	MBSPA satisfies  	
	\begin{alignat}{6}
	&\EE\left[\dist(0,\G_{\bar{\gamma}}(\bar{w}^R))\right]&\leq&\sqrt{\frac{(L+N^{\theta})}{N}
		\left(\DD+\frac{2l^2}{N^{\theta}}\right)}+\sqrt{\frac{6\sigma^2}{\lceil
			N^{\alpha}\rceil}}+\frac{2l}{N^{\theta}}\left(2N^{\tau}+L\right),\nonumber
	\end{alignat}	
	where $\DD=4(\Phi(w^1)-\Phi(w^*))$ and $w^*$ is a global minimizer of $\Phi(\cdot)$.		
\end{theorem}
\begin{proof}	
	We first verify that $\bar{\gamma}=\frac{1}{N^{\tau}}\geq 
	\frac{1}{N^{\theta}}\geq\frac{1}{L+N^{\theta}}=\gamma$. From Property \ref{eq:switch}, taking 
	$\zeta^{\lambda}(w^R)=\bar{w}^R$,
	$$\dist(0,\G_{\bar{\gamma}}(\bar{w}^R))\leq 	
	||\G^R_{\gamma, E}(w^R)||_2+2l\lambda\left(\frac{2}{\bar{\gamma}}+L\right).$$	
	Taking its expectation, 	
	\begin{alignat}{6}
	&\EE\left[\dist(0,\G_{\bar{\gamma}}(\bar{w}^R))\right]&\leq& 
	\EE[||\G^R_{\gamma, E}(w^R)||_2]+2l\lambda\left(\frac{2}{\bar{\gamma}}+L\right)\nonumber\\
	&&\leq& 
	\sqrt{\EE\left[||\G^R_{\gamma, 
			E}(w^R)||^2_2\right]}+\frac{2l}{N^{\theta}}\left(2N^{\tau}+L\right)\nonumber\\
	&&\leq&\sqrt{\frac{(L+N^{\theta})}{N}\tilde{\DD}}+\sqrt{\frac{6\sigma^2}{\lceil 			
	N^{\alpha}\rceil}}+\frac{2l}{N^{\theta}}\left(2N^{\tau}+L\right),\nonumber
	\end{alignat}	
	where the second inequality uses Jensen's inequality and the third inequality follows from 
	Lemma \ref{th:1}.
	The result then follows using Property \ref{le:2} as  	
	\begin{alignat}{6}
	&&\tilde{\DD}=4(\tPhi_{\lambda}(w^1)-\tPhi_{\lambda}(w^*_{\lambda}))&\leq 4(\Phi(w^1)-\Phi(w^*)) + 
	2l^2\lambda\nonumber\\
	&&&=\DD+\frac{2l^2}{N^{\theta}}.\nonumber
	\end{alignat} 	
\end{proof}
Having bounded the expected distance of $\G_{\bar{\gamma}}(\bar{w}^R)$ from the origin, we 
prove an $\epsilon$-accurate point convergence complexity.
\begin{corollary}	
	\label{mbcomp}
	Assume that $g(w)$ is Lipschitz continuous with 
	parameter $l$. To obtain an $\epsilon$-accurate solution \eqref{epsol} using MBSPA, the 
	gradient call complexity is $O(\epsilon^{-5})$ and the proximal operator complexity is 
	$O(\epsilon^{-3})$ choosing $\theta=\frac{1}{3}$, 
	$\alpha=\frac{2}{3}$, and $\tau=0$. 	
\end{corollary}
\begin{proof}	
	From Theorem \ref{mbconv},	
	\begin{alignat}{6}
	&\EE\left[\dist(0,\G_{\bar{\gamma}}(\bar{w}^R))\right]&\leq&\sqrt{\frac{(L+N^{\theta})}{N}
		\left(\DD+\frac{2l^2}{N^{\theta}}\right)}+\sqrt{\frac{6\sigma^2}{\lceil
			N^{\alpha}\rceil}}+\frac{2l}{N^{\theta}}\left(2+L\right)\nonumber\\
	&&=&O(N^{0.5\theta-0.5})+O(N^{-0.5\alpha})+O(N^{-\theta}).\nonumber	
	\end{alignat}	
	Setting $\theta=\frac{1}{3}$ and $\alpha=\frac{2}{3}$,   	
	\begin{alignat}{6}
	&&\EE\left[\dist(0,\G_{\bar{\gamma}}(\bar{w}^R))\right]&\leq 
	O(N^{-\frac{1}{3}}).\nonumber
	\end{alignat}	
	An $\epsilon$-accurate solution will require less than $N=O(\epsilon^{-3})$ 
	iterations. Two proximal operations are required per iteration, which establishes the proximal 
	operator complexity of $O(\epsilon^{-3})$. The number of gradient calls per iteration is $\lceil 
	N^\alpha\rceil=O(\epsilon^{-2})$. The number of gradient calls to get an 
	$\epsilon$-accurate solution is then	
	\begin{alignat}{6}
	&N\lceil N^\alpha\rceil=O(\epsilon^{-5}).\nonumber
	\end{alignat}	
\end{proof}

\section{Variance reduced stochastic proximal algorithm for finite-sum problems}
In this section we assume that 
$$f(w)=\frac{1}{n}\sum_{j=1}^nf_j(w),$$
where each $f_j(w)$ is $L$-smooth. 
\begin{algorithm}
	\caption{Variance reduced stochastic proximal algorithm (VRSPA)}
	\label{alg:two}
	\begin{algorithmic}
		\STATE {\bfseries Input:} $\tilde{w}^1\in \RR^d$, $N\in\ZZ_{>0}$, $\alpha,\theta\in \RR$
		\STATE $m=\lceil n^{\alpha}\rceil$, $b=m^2$ 
		\STATE $S=\lceil\frac{N}{m}\rceil$, $\lambda=(Sm)^{-\theta}$
		\STATE $L_{\lambda}=L+\frac{1}{\lambda}$, $\gamma=\frac{1}{6L_{\lambda}}$
		\STATE $R\sim \text{uniform}\{1,...,S\}$
		\FOR{$k=1,2,...,R$}
		\STATE $w^{k}_1=\tilde{w}^{k}$
		\STATE $G^k=\nabla f(\tilde{w}^{k})$
		\FOR{$t=1,2,...,m$}
		\STATE $\zeta^{\lambda}(w^k_t)\in \prox_{\lambda g}(w^k_t)$
		\STATE $I\sim \text{uniform}\{1,...,n\}^b$
		\STATE $V^k_t=\frac{1}{b}\sum_{j\in I}\left(\nabla f_j(w^k_t)-\nabla 		
		f_j(\tilde{w}^k)\right)+G^k+\frac{1}{\lambda}(w^k_t-\zeta^{\lambda}(w^k_t))$
		\STATE $w^k_{t+1}=\prox_{\gamma h}(w^k_t-\gamma V^k_t)$
		\ENDFOR		
		\STATE $\tilde{w}^{k+1}=w^k_{m+1}$
		\ENDFOR
		\STATE $T\sim \text{uniform}\{1,...,m\}$		
		\STATE {\bfseries Output:} $\bar{w}^R_T\in \prox_{\lambda g}(w^R_T)$
	\end{algorithmic}
\end{algorithm}
\subsection{Convergence analysis}
We require the function $E^{k,t}_{\lambda}(w)$ in our convergence analysis, which is 
constructed in the same way as $E^k_{\lambda}(w)$ \eqref{Efunc}, but using $w^k_t$ instead of $w^k$. 
This function possesses the same characteristics as found in Property \ref{Eprop}. In addition, let 
\begin{alignat}{6}
&\G^{k,t}_{\gamma, E}(w^k_t):=\P_{\gamma}(w^k_t,\nabla E^{k,t}_{\lambda}(w^k_t))\nonumber.
\end{alignat}
The convergence analysis follows the work of \cite{li2018} adapted to our problem. The proof of Lemma 
\ref{vrlemma} can be found in Section 5 of the supplementary material.
\begin{lemma}	
	\label{vrlemma}
	For an initial value $\tilde{w}_1$, $N\in\ZZ_{>0}$, and $\alpha,\theta\in\RR$, VRSGA 
	generates $w^R_T$ satisfying the following bound.	  	  	
	\begin{alignat}{6} 
	&&\EE\left[||\G^{R,T}_{\gamma, 
		E}(w^R_T)||^2_2\right]&\leq\tilde{\DD}\frac{L+(Sm)^{\theta}}{Sm}\nonumber
	\end{alignat}	
	where $\tilde{\DD}=36(\tPhi_{\lambda}(\tilde{w}^1)-\tPhi_{\lambda}(w^*_{\lambda}))$ and 
	$w^*_{\lambda}$ is a global minimizer of $\tPhi_{\lambda}(\cdot)$.	
\end{lemma}
\begin{theorem}	
	\label{vrconv}	
	Assume that $g(w)$ is Lipschitz continuous with parameter $l$ and 
	$\bar{\gamma}=\frac{1}{N^{\tau}}$ for $\tau\leq\theta$. The output $\bar{w}^R_T$ 
	of VRSPA satisfies the following inequality. 	
	\begin{alignat}{6}
	&\EE\left[\dist(0,\G_{\bar{\gamma}}(\bar{w}^R_T))\right]&\leq&\sqrt{\frac{\left(L+(Sm)^{\theta}\right)
			\left(\DD+18l^2(Sm)^{-\theta}\right)}{Sm}}+\frac{2l}{(Sm)^{\theta}}\left(2N^{\tau}+L\right),\nonumber
	\end{alignat}	
	where $\DD=36(\Phi(w^1)-\Phi(w^*))$ and $w^*$ is a global minimizer of $\Phi(\cdot)$.		
\end{theorem}
\begin{corollary}	
	\label{fscomp}	
	Assume that $g(w)$ is Lipschitz continuous with parameter $l$. To obtain an $\epsilon$-accurate 
	solution \eqref{epsol} using VRSPA, the gradient call complexity is 
	$O(n^{\frac{2}{3}}\epsilon^{-3})$ and the proximal operator complexity is $O(\epsilon^{-3})$ 
	choosing $\alpha=\theta=\frac{1}{3}$, and $\tau=0$. 		
\end{corollary}
The proofs of Theorem \ref{vrconv} and Corollary \ref{fscomp} are similar to those of Theorem 
\ref{mbconv} and Corollary \ref{mbcomp}, and can be found in Section 6 of the supplementary 
material.	
\section{Experiments}
\label{NE}
We conducted experiments comparing our algorithms to SDCAM \citep{liu2017} for the problem of 
non-negative sparse PCA \eqref{eq:ap1} on datasets MNIST 
\citep{lecun1998} and RCV1 \citep{lewis2004}. The dimensions of 
MNIST are $n=60,000$ and $d=784$, and those of RCV1 are $n=804,414$ and $d=47,236$. All experiments 
were conducted using MATLAB 2017b on a Mac Pro with a 2.7 GHz 
12-core Intel Xeon E5 processor and 64GB of RAM. In Figures \ref{T1} we compare the 
performance of all algorithms, plotting the objective function versus wall-clock 
time. The values for $\alpha$ and $\theta$ established in Corollaries \ref{mbcomp} and \ref{fscomp} 
were used to implement MBSPA and VRSPA. 
It was hypothesized that the inferior performance of VRSPA was due to its smaller stepsize, so VRSPA2 
is VRSPA using the stepsize of MBSPA. All parameters of SDCAM were left unchanged as used in the 
available implementation\footnote{\url{http://www.mypolyuweb.hk/~tkpong/Matrix_sparse_MP_codes/}}. 
The regularizer's parameters were chosen as $\kappa=\frac{1}{d}$ and $\nu=1$. We observe that our 
algorithms were able to achieve faster convergence in both experiments. 

\pgfplotsset{every axis title/.append style={at={(0.5,0.95)}}}
\pgfplotsset{every tick label/.append style={font=\footnotesize}}
\begin{figure}[H]
	\centering
	\begin{tikzpicture}
	\begin{groupplot}[group style={group name=myplot,group size= 2 by 1},height=6cm,width=7.4cm]
	\nextgroupplot[title=MNIST,label 
	style={font=\normalsize},xlabel=time (s),ylabel=$\Phi(w)$,y label style={at={(axis description 
	cs:0.04,.5)}},xtick={0,1,...,7},ytick={-8,-10,...,-18}]
	\addplot[line width=2pt,draw=blue]
	table[x=x,y=y]
	{MBSPAm.dat};\label{plots:plot1}
	\addplot[line width=2pt,draw=red]
	table[x=x,y=y]
	{VRSPAm.dat};\label{plots:plot2}
	\addplot[line width=2pt,draw=cyan]
	table[x=x,y=y]
	{VRSPA2m.dat};\label{plots:plot3}		
	\addplot[line width=2pt,draw=black,dotted]
	table[x=x,y=y]
	{SDCAMm.dat};\label{plots:plot4}
	\nextgroupplot[title=RCV1,label 
	style={font=\normalsize},xlabel=time (s), 
	ytick={-0.001,-0.003,...,-0.011},xtick={0,2000,...,8000}]
	\addplot[line width=2pt,draw=blue]
	table[x=x,y=y]
	{MBSPAr.dat};\label{plots:plot1}
	\addplot[line width=2pt,draw=red]
	table[x=x,y=y]
	{VRSPAr.dat};\label{plots:plot2}
	\addplot[line width=2pt,draw=cyan]
	table[x=x,y=y]
	{VRSPA2r.dat};\label{plots:plot3}			
	\addplot[line width=2pt,draw=black,dotted]
	table[x=x,y=y]
	{SDCAMr.dat};\label{plots:plot4}
	\end{groupplot}
	\path (myplot c1r1.south west|-current bounding box.south)--
	coordinate(legendpos)
	(myplot c2r1.south east|-current bounding box.south);
	\matrix[matrix of nodes,anchor=south,draw,inner sep=0.2em,draw]at([yshift=-0.5cm]legendpos)
	{\ref{plots:plot1}&MBSPA&[5pt]
	 \ref{plots:plot2}&VRSPA&[5pt]
	 \ref{plots:plot3}&VRSPA2&[5pt]
	 \ref{plots:plot4}&SDCAM&[5pt]\\};
	\end{tikzpicture}
	\caption{Comparison of algorithms of this paper and SDCAM \citep{liu2017} on datasets MNIST and 
	RCV1.} 
	\label{T1}
\end{figure}
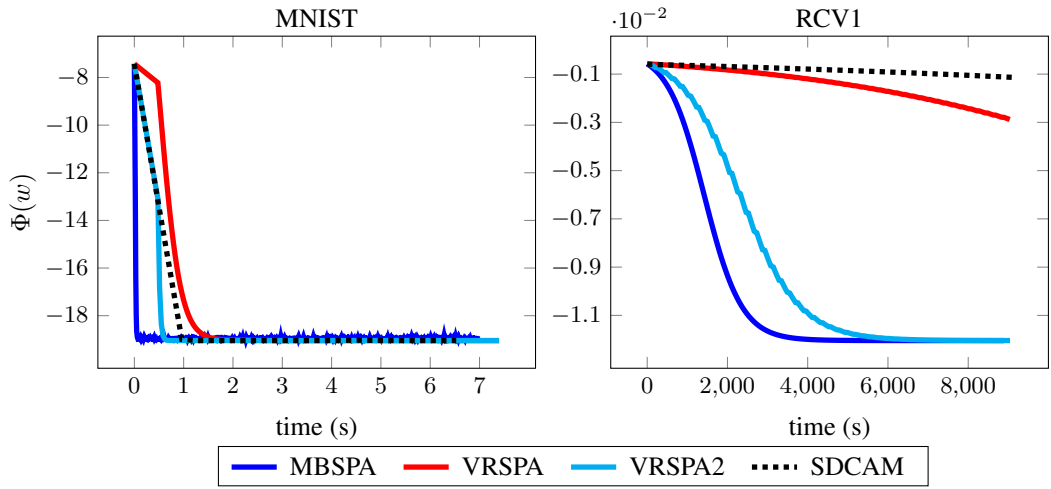

\section{Conclusion}
In this paper we considered minimizing a smooth non-convex loss function with a non-smooth non-convex 
regularizer with convex constraints. We presented two stochastic proximal gradient algorithms, and to 
the best of our knowledge, the first non-asymptotic convergence results for this class of problem. 
The convergence complexities in this paper are equal to or superior to the results found in 
\citep{xu2018} and \citep{metel2019} which consider the case of our problem setting when $h(w)=0$. 
In an empirical study we found our algorithms to converge faster than a state-of-the-art deterministic 
algorithm. 
\bibliography{PSGM}

\end{document}



\title{Stochastic Proximal Methods for Non-Smooth Non-Convex Constrained Sparse Optimization: 
Supplementary Material}

\maketitle

\section{Applications of non-smooth non-convex constrained sparse optimization}

All applications presented here optimize over a closed convex feasible region $C$. We take $h(w)$ as 
the indicator 
function $\delta_{C}(w)$, 
$$\delta_{C}(w)= 
\begin{cases}
0 & \text{ if } w\in C\\ 
\infty & \text{ otherwise. }   
\end{cases}$$
For a nonempty set $C$, $\prox_{\delta_C}(w)=\argmin_{u\in C}||u-w||_2$ is the projection onto the 
set, see for example \citep[Theorem 6.24]{beck2017}. 

{\em Sparse binary classification with outlier detection and fairness constraints:}

We are given training data $\{x,y\}$ where $y=\{y^1,y^2,...,y^n\}$, $y^j\in\{-1,1\}$ is the label set, 
and $x=\{x^1,x^2,...,x^n\}$, $x^j\in \RR^{d'}$ is the feature set. In the application of classifying 
people, there may be sensitive attributes such as race or sex. Even if a sensitive attribute $x_a\in 
x$ is removed from the feature set, our predictions may still be correlated to it, resulting in our 
model disproportionally treating a subset of the population unfairly. This is remedied by bounding the 
covariance between the sensitive attribute $x_a$ and the model output as done in \citep{zafar2017}, 
$$\frac{1}{n}\left|\sum_{j=1}^n(x_a^j-\bar{x}_a)v^Tx_{-a}^j\right|\leq c,$$

where $\bar{x}_a$ is the mean of $x_a$, $x^j_{-a}$ is the $j^{th}$ 
feature vector with the sensitive attribute removed, $v^Tx_{-a}^j$ is our model output using 
decision variables $v\in \RR^{d'}$, and $c>0$ determines the maximum covariance tolerated.

We consider the smoothed 0-1 loss of \cite{zhao2010} as our loss function,
\begin{alignat}{6}\label{eq:lossf}
\LL(u)=\begin{cases}
0 & \text{ if } u>1 \\ 
\frac{1}{4}u^3-\frac{3}{4}u+\frac{1}{2} & \text{ if } -1\leq u\leq 1 \\ 
1  & \text{ otherwise }.\\ 
\end{cases}
\end{alignat}
We implement outlier detection by the mean-shift method, modifying our prediction to 
$v^Tx_{-a}^j+z^j$, using decision variables $z\in \RR^n$ to reduce the loss incurred by outliers. 
It was shown in \citep{she2011} that the $l_1$ norm is not effective as a penalizer of $z$ when 
multiple outliers are present, which motivates the use of a non-convex regularizer. As all of the 
regularizers considered for $g(w)$ in Section 1 of the main text are separable, we are able to 
take $g(w)=g^1(v)+g^2(z)$, and are free to use different regularizers for $v$ and $z$. The 
classification problem is then solved by the following minimization,
\begin{alignat}{6}\label{eq:ap3}
&\min\limits_{v,z}&&\text{ }\frac{1}{n}\sum_{j=1}^n\LL(y^j(v^Tx_{-a}^j+z^j))+g^1(v)+g^2(z)\\
&\text{s.t.}&&\frac{1}{n}\left|\sum_{j=1}^n(x_a^j-\bar{x}_a)v^Tx_{-a}^j\right|\leq c.\nonumber
\end{alignat}
The feasible region of decision variables $v$ can be rewritten as  
$$C=\left\{v :  
\begin{matrix}
\hat{x}^Tv\leq c\\
-\hat{x}^Tv\leq c
\end{matrix}
\right\},$$
where $\hat{x}=\frac{1}{n}\sum_{j=1}^n(x_a^j-\bar{x}_a)x_{-a}^j$. The projection onto $C$ 
can be computed as 
$$P_C(v)=\begin{cases}
v-\frac{\hat{x}^Tv-c}{||\hat{x}||^2_2}\hat{x} & \text{ if } \hat{x}^Tv > c \\ 
v-\frac{\hat{x}^Tv+c}{||\hat{x}||^2_2}\hat{x} & \text{ if } -\hat{x}^Tv > c \\ 
v & \text{ else, } \\ 
\end{cases}$$
which uses the projection onto a halfspace \citep[Lemma 6.26]{beck2017}. We take $g^1(v)$ equal to 
MCP. This function is separable with $g^1(v):=\sum_{i=1}^{d'}g^1_i(v_i)$, where for 
$\kappa_1, \nu_1>0$,
$$g^1_i(v_i)=\kappa_1\int_0^{|v_i|}\max\left(0,1-\frac{u}{\nu_1\kappa_1}\right)du= 
\begin{cases}
\kappa_1|v_i| - \frac{v^2_i}{2\nu_1} & \text{ if } |v_i|\leq \nu_1\kappa_1 \\ 
\nu_1\kappa_1^2/2 & \text{ if } |v_i|>\nu_1\kappa_1  
\end{cases}$$ 
\begin{property}
	\label{firstprop}	
	$g^1(v)$ is $\kappa_1\sqrt{d'}$-Lipschitz continuous.
\end{property}
\begin{proof}
	Assume $v_i\geq 0$, over which $g^1_i(v_i)$ is differentiable and 
	$\left|\frac{d g^1_i}{v_i}(v_i)\right|\leq \kappa_1$. Using the mean value theorem, for $u_i\geq 
	0$, $|g^1_i(u_i)-g^1_i(v_i)|\leq \kappa_1|u_i-w_i|$. Given the symmetry of $g^1_i(\cdot)$, this 
	bound holds for all $u_i,w_i$, and		
	\begin{alignat}{6}
	&&|g^1(u)-g^1(w)|&=\left|\sum_{i=1}^{d'}(g^1_i(u_i)-g^1_i(w_i))\right|\nonumber\\
	&&&\leq \sum_{i=1}^{d'}|g^1_i(u_i)-g^1_i(w_i)|\nonumber\\
	&&&\leq \kappa_1\sum_{i=1}^{d'}|u_i-w_i|\nonumber\\
	&&&\leq \kappa_1\sqrt{d'}||u-w||_2\nonumber
	\end{alignat}	
\end{proof}
As considered in \citep{she2011}, we set $g^2(z)$ equal to SCAD, which is also separable. For 
$\kappa_2>0$ and $\nu_2>2$,
$$g^2_i(z_i)=\kappa_2\int_0^{|z_i|}\min\left(1,\frac{\max(0,\nu_2\kappa_2-u)}{(\nu_2-1)\kappa_2}\right)du=
\begin{cases}
\kappa_2|z_i| & \text{ if } |z_i|\leq \kappa_2 \\
\frac{-z^2_i+2\nu_2\kappa_2|z_i|-\kappa_2^2}{2(\nu_2-1)} & \text{ if } \kappa_2<|z_i|\leq 
\nu_2\kappa_2\\ 
(\nu_2+1)\kappa_2^2/2 & \text{ if } |z_i|>\nu_2\kappa_2.  
\end{cases}$$
Similarly to MCP, SCAD is symmetric and $|\dist(0,\partial g^2_i(z_i))|\leq \kappa_2$. Using the 
same reasoning as in the proof of Property \ref{firstprop}, we get the following property.
\begin{property}
	\label{thirprop}	
	$g^2(z)$ is $\kappa_2\sqrt{n}$-Lipschitz continuous.
\end{property}
For the closed form solutions of the proximal operators of MCP and SCAD see \citep{gong2013}.
\begin{property}	
	\label{secprop}
	The function $\frac{1}{n}\sum_{j=1}^n\LL(y^j(v^Tx_{-a}^j+z^j))$, with $\LL(\cdot)$ as defined in 
	\eqref{eq:lossf}, is $\frac{3}{2n}\sum_{j=1}^n||[(x_{-a}^j)^T,1]||^2_2$-smooth in $v$ and $z$.	
\end{property}
\begin{proof}	
	We see that $|\LL''(u)|\leq \frac{3}{2}$. Using the mean value theorem, 
	$|\LL'(u)-\LL'(t)|\leq \frac{3}{2}|u-t|$. Composing it with the affine function 
	$y^j(v^Tx_{-a}^j+z^j)$, the resulting function is $\frac{3}{2}||[(x_{-a}^j)^T,1]||^2_2$-smooth 
	\citep[Claim 12.9]{shalev2014}. We conclude that 
	$\frac{1}{n}\sum_{j=1}^n\LL(y^j(v^Tx_{-a}^j+z^j))$ 
	is $\frac{3}{2n}\sum_{j=1}^n||[(x_{-a}^j)^T,1]||^2_2$-smooth.	
\end{proof}

In the following two applications, it is assumed that $g(w)$ is taken as either MCP or SCAD.

{\em Sparse portfolio optimization using S-shaped utility with loss aversion:} 

We assume there are $d$ risky assets with stochastic returns $r_i$, $i=1,...,d$, and an investor 
desires to place a fraction $w_i$ of their wealth into each asset $i$. Finding a sparse portfolio is 
desirable as trading fewer assets results in fewer transaction costs. Motivated by prospect theory 
\citep{kahneman1979}, we assume the investor is risk adverse in gains (concave utility) and risk 
seeking in losses (convex utility). Our objective is to maximize the following 
exponential utility function
\begin{alignat}{6}\label{eq:eu}
F(w,r)= 
\begin{cases}
\frac{1-e^{-\psi^1(\sum_{i=1}^dw_ir_i)}}{\psi^1} & \text{ if } \sum_{i=1}^dw_ir_i\geq 0\\ 
\frac{e^{\psi^2(\sum_{i=1}^dw_ir_i)}-1}{\psi^2} & \text{ otherwise, }   
\end{cases}
\end{alignat} 
where $\psi^1,\psi^2>0$. This utility function has been considered in 
\citep{kobberling2005,pirvu2012}. Choosing 
$\psi^1>\psi^2$ models loss aversion, where the investor has increased sensitivity to losses than to 
gains. Our optimization problem is then 
\begin{alignat}{6}\label{eq:ap2}
&\max&&\text{ }\EE_r[F(w,r)]+g(w)\\
&\text{s.t.}&&\sum_{i=1}^dw_i\leq 1, w\geq 0,\nonumber
\end{alignat}
where we assume there should be no short selling. In order to project onto the constraints let 
$w_{d+1}=1-\sum_{i=1}^d w_i$ and $C=\{w:\sum_{i=1}^{d+1}w_i=1, w\geq 0\}$.  
The projection onto the probability simplex can be achieved using a simple 
non-iterative algorithm such as found in \citep{wang2013}.

We assume we have access to $n$ historical observations of $r$, $r^j$ for 
$j=1,...,n$. We take a distribution-free approach, optimizing directly over the observations,
\begin{alignat}{6}
f(w)=\frac{1}{n}\sum_{j=1}^n F(w,r^j)\label{ef}
\end{alignat}
\begin{property}	
	\label{thirdprop}
	The function $f(w)=\frac{1}{n}\sum_{j=1}^n F(w,r^j)$, where $F(w,r^j)$ is as defined in 
	\eqref{eq:eu} is $\frac{\max(\psi^1,\psi^2)}{n}\sum_{j=1}^n||r^j||^2_2$-smooth.	
\end{property}
\begin{proof}	
	We first consider the univariate function		
	\begin{alignat}{6}
	\hat{F}(u)= 
	\begin{cases}
	\hat{F}_1(u)=\frac{1-e^{-\psi^1u}}{\psi^1} & \text{ if } u\geq 0\\ 
	\hat{F}_2(u)=\frac{e^{\psi^2u}-1}{\psi^2} & \text{ otherwise }   
	\end{cases}
	\end{alignat} 	
	The first and second derivatives are 
	$$\hat{F}'_1(u)=e^{-\psi^1u}\quad\hat{F}''_1(u)=-\psi^1e^{-\psi^1u}\quad
	\hat{F}'_2(u)=e^{\psi^2u}\quad\hat{F}''_2(u)=\psi^2e^{\psi^2u}$$	
	We can see that $|\hat{F}''_1(u)|\leq \psi^1$ and $|\hat{F}''_2(u)|\leq \psi^2$ over their 
	domains. Assume that $w,x\geq 0$ and $v,u\leq 0$. Using the mean value theorem, 
	\begin{alignat}{6}
	|\hat{F}'_1(w)-F_1'(x)|\leq \psi^1|w-x|\label{l1}\\
	|F_2'(v)-F_2'(u)|\leq \psi^2|v-u|\label{l2}
	\end{alignat}	
	Given $w$ and $v$, we can take 
	$u=\frac{-\psi^1}{\psi^2}w$ with $\hat{F}'_2(u)=\hat{F}'_1(w)$ and 
	$x=\frac{-\psi^2}{\psi^1}v$ 
	with $\hat{F}'_1(x)=\hat{F}'_2(v)$.  
	Assume $\psi^1\geq\psi^2$, then 
	\begin{alignat}{6}
	|\hat{F}'_1(w)-\hat{F}_2'(v)|&=|\hat{F}'_1(w)-\hat{F}_1'(x)|\nonumber\\
	&\leq \psi^1|w-x|\nonumber\\
	&\leq \psi^1|w-v|.\label{l3}
	\end{alignat}	
	Assuming now $\psi^2\geq\psi^1$, then 
	\begin{alignat}{6}
	|\hat{F}'_1(w)-\hat{F}_2'(v)|&=|\hat{F}_2'(u)-\hat{F}_2'(v)|\nonumber\\
	&\leq \psi^2|u-v|\nonumber\\
	&\leq \psi^2|w-v|,\label{l4}
	\end{alignat}	
	From \eqref{l1}-\eqref{l4}, we conclude that 
	$\hat{F}(u)$ is $\max(\psi^1,\psi^2)$-smooth. As shown in the proof of Property \ref{secprop}, 
	since $F(w,r^j)$ is $\hat{F}(u)$ composed with the 
	affine function $\sum_{i=1}^dw_ir^j_i$, it is $||r^j||^2_2\max(\psi^1,\psi^2)$-smooth and $f(w)$ 
	is $\frac{\max(\psi^1,\psi^2)}{n}\sum_{j=1}^n||r^j||^2_2$-smooth.	
\end{proof}

{\em Non-negative sparse principal component analysis:} 

The projection onto $C=\{w: ||w||_2\leq 1, w\geq 0\}$ has the explicit solution \citep[Theorem 
7.1]{bauschke2018} 
$$P_C(w)=\frac{\max(w,0)}{\max(||\max(w,0)||_2,1)}.$$
\begin{property}	
	\label{fourthprop}
	Given a dataset $x\in \RR^{d\times n}$, the function $f(w)=-\frac{1}{2n}\sum_{j=1}^n(w^Tx_j)^2$ is 
	$\frac{1}{n}||xx^T||_2$-smooth, where $||\cdot||_2$ is the spectral norm.	
\end{property}
\begin{proof}
We can rewrite the function $f(w)$,
$$f(w)=-\frac{1}{2n}\sum_{j=1}^n(w^Tx_j)^2=-\frac{1}{2n}w^T\left(\sum_{j=1}^nx_jx_j^T\right)w=-\frac{1}{2n}w^Txx^Tw.$$
\begin{alignat}{6}
&\left\lVert\nabla f(w)-\nabla f(w')\right\lVert_2&=&\bigg\lVert 
-\frac{1}{n}xx^Tw+\frac{1}{n}xx^Tw'\bigg\lVert_2\nonumber\\
&&\leq&\frac{1}{n}\big\lVert xx^T\big\lVert_2\big\lVert w-w'\big\lVert_2.\nonumber
\end{alignat}		
\end{proof}
VRSPA requires that each $f_j(w)$ be $L$-smooth, so in our numerical experiments we took 
$L=\max\limits_j\big\lVert x_jx^T_j\big\lVert_2$. A similar approach can be taken when implementing 
the other applications.
\section{Proof of Property 1}
\begin{customproperty}{1}
	The following holds for $E^k_{\lambda}(w)$.		
	\begin{alignat}{6}
	&E_{\lambda}^k(w)+h(w)\geq \tPhi_{\lambda}(w)\hspace{0 mm} \text{ for all 
	}w\in\RR^d\label{eq:e1}\\
	&E_{\lambda}^k(w^k)+h(w^k)=\tPhi_{\lambda}(w^k)\label{eq:e3}\\
	&E_{\lambda}^k(w) \text{ is } 
	L_{\lambda}:=\left(L+\frac{1}{\lambda}\right)-\text{smooth.}\label{eq:e2}
	\end{alignat}	
\end{customproperty}
\begin{proof}
	The Moreau envelope can be written as a difference of convex functions, 
	\begin{alignat}{6}
	&e_{\lambda}g(w)&&=\frac{1}{2\lambda}||w||^2_2-D^{\lambda}(w).\label{fsg}
	\end{alignat}
	As the supremum of a set of affine functions, $D^{\lambda}(w)$ (14) is convex, and 
	$\zeta^{\lambda}(w)$ 
	attains the supremum of $D^{\lambda}(w)$ from (8). The difference between 
	$E_{\lambda}^k(w)+h(w)$ and $\tPhi_{\lambda}(w)$ is the difference between $U^k_{\lambda}(w)$ and 
	$e_{\lambda}g(w)$, so we only focus on showing that \eqref{eq:e1} and \eqref{eq:e3} hold between 
	these two terms.
	
	\eqref{eq:e1}: As found in \citep{liu2017}, for any $w,z\in \RR^d$, 	
	\begin{alignat}{6}
	&D^{\lambda}(w)-D^{\lambda}(z)&=&\sup_{x\in 		
		\RR^d}\left(\frac{1}{\lambda}w^Tx-\frac{1}{2\lambda}||x||^2_2-g(x)\right)-\sup_{x\in 
		\RR^d}\left(\frac{1}{\lambda}z^Tx-\frac{1}{2\lambda}||x||^2_2-g(x)\right)\nonumber\\
	&&\geq&\frac{1}{\lambda}w^T\zeta^{\lambda}(z)-\frac{1}{2\lambda}||\zeta^{\lambda}(z)||^2_2
	-g(\zeta^{\lambda}(z))\nonumber\\	
	&&&-\left(\frac{1}{\lambda}z^T\zeta^{\lambda}(z)-\frac{1}{2\lambda}||\zeta^{\lambda}(z)||^2_2-g(\zeta^{\lambda}(z))\right)\nonumber\\
	&&=&\frac{1}{\lambda}\zeta^{\lambda}(z)^T(w-z).\nonumber
	\end{alignat}	
	Setting $z=w^k$,  	
	\begin{alignat}{6}
	&e_{\lambda}g(w)&&=\frac{1}{2\lambda}||w||^2_2-D^{\lambda}(w)\nonumber\\
	&&&\leq\frac{1}{2\lambda}||w||^2_2-(D^{\lambda}(w^k)+\frac{1}{\lambda}\zeta^{\lambda}(w^k)^T(w-w^k))\nonumber\\
	&&&=U^k_{\lambda}(w).\nonumber
	\end{alignat}	
	
	\eqref{eq:e3}:  
	$U^k_{\lambda}(w^k)=\frac{1}{2\lambda}||w^k||^2_2-D^{\lambda}(w^k)=e_{\lambda}g(w^k)$ from 
	\eqref{fsg}.\\
	
	\eqref{eq:e2}: \begin{alignat}{6}
	&\left\lVert\nabla E_{\lambda}^k(w)-\nabla E_{\lambda}^k(w')\right\lVert_2&=&\bigg\lVert\nabla 
	f(w)+\frac{1}{\lambda}\left(w-\zeta^{\lambda}(w^k)\right)-\left(\nabla	
	f(w')+\frac{1}{\lambda}\left(w'-\zeta^{\lambda}(w^k)\right)\right)\bigg\lVert_2\nonumber\\
	&&\leq&(L+\frac{1}{\lambda})\lVert w-w'\rVert_2.\nonumber
	\end{alignat}	
\end{proof}

\section{Proof of Lemma 2}
\begin{customlemma}{2}
	For an initial value $w_1\in \RR^d$, $N\in\ZZ_{>0}$, and $\alpha,\theta \in \RR$, MBSPA 
	generates 
	$w^R$ satisfying the following bound.  	
	\begin{alignat}{6}
	&&\EE||\G^R_{\gamma, E}(w^R)||^2_2&\leq
	\frac{(L+N^{\theta})}{N}\tilde{\DD}+\frac{6}{\lceil
		N^{\alpha}\rceil}\sigma^2\nonumber
	\end{alignat}
	where $\tilde{\DD}=4(\tPhi_{\lambda}(w^1)-\tPhi_{\lambda}(w^*_{\lambda}))$ and 
	$w^*_{\lambda}$ is a global minimizer of $\tPhi_{\lambda}(\cdot)$.
\end{customlemma}
In order to prove this result, we require the following properties. The proof of Property 
\ref{eq:minib} can be found in \citep{metel2019}, which we include here for completeness.
\begin{property}	
	\label{eq:minib}	
	\begin{alignat}{6}
	&&\EE||\nabla A^k_{\lambda, M}(w^k,\xi^k)-\nabla 
	E^k_{\lambda}(w^k)||^2_2&\leq\frac{\sigma^2}{M}\nonumber
	\end{alignat}	
\end{property}
\begin{proof}
	From (15) and (16), $\nabla A^k_{\lambda, M}(w^k,\xi^k)-\nabla 	
	E^k_{\lambda}(w^k)=\frac{1}{M}\sum_{j=1}^{M}\nabla F(w^k,\xi^k_j)-\nabla 
	f(w^k)$. Taking the expectation of its squared norm,
	\begin{alignat}{6}
	&&\EE||\nabla A^k_{\lambda, M}(w^k,\xi^k)-\nabla 
	E^k_{\lambda}(w^k)||^2_2&=\EE||\frac{1}{M}\sum_{j=1}^{M}(\nabla F(w^k,\xi^k_j)-\nabla 
	f(w^k))||^2_2\nonumber\\
	&&&=\frac{1}{M^2}\EE\sum_{i=1}^n\left(\sum_{j=1}^M\nabla F(w^k,\xi^k_j)_i-\nabla 
	f(w^k)_i\right)^2.\nonumber
	\end{alignat}	
	For $j\neq l$, $\nabla F(w^k,\xi^k_j)_i-\nabla f(w^k)_i$ and $\nabla F(w^k,\xi^k_l)_i-\nabla 
	f(w^k)_i$ 
	are independent random variables with zero mean. It follows that    
	\begin{alignat}{6}
	&\EE[(\nabla F(w^k,\xi^k_j)_i-\nabla f(w^k)_i)(\nabla F(w^k,\xi^k_l)_i-\nabla 
	f(w^k)_i)]=\nonumber\\
	&\EE[(\nabla F(w^k,\xi^k_j)_i-\nabla f(w^k)_i)]\EE[(\nabla F(w^k,\xi^k_l)_i-\nabla 
	f(w^k)_i)]=0,&\nonumber
	\end{alignat}	
	and 	
	\begin{alignat}{6}
	&&\frac{1}{M^2}\EE\sum_{i=1}^n\left(\sum_{j=1}^M\nabla F(w^k,\xi^k_j)_i-\nabla 
	f(w^k)_i\right)^2&=\frac{1}{M^2}\EE\sum_{i=1}^n\sum_{j=1}^{M}(\nabla F(w^k,\xi^k_j)_i-\nabla 
	f(w^k)_i)^2\nonumber\\
	&&&=\frac{1}{M^2}\sum_{j=1}^{M}\EE||\nabla F(w^k,\xi^k_j)-\nabla 
	f(w^k)||^2_2\leq\frac{\sigma^2}{M}\nonumber
	\end{alignat}	
	using (6).
\end{proof}
The proof of Property \ref{eq:psize} can be found in \citep{ghadimi2016}, which we include in our 
notation here for clarity. 
\begin{property}	
	\label{eq:psize}
	Let $w,s\in\RR^d$ and $\gamma>0$, then	
	\begin{alignat}{6}
	&-\langle s,\P_{\gamma}(w,s)\rangle\leq 
	\frac{1}{\gamma}\left(h(w)-h(\prox_{\gamma h}(w-\gamma 
	s))\right)-||\P_{\gamma}(w,s)||^2_2.\nonumber 	
	\end{alignat}	
\end{property}
\begin{proof}
	By the optimality of $\prox_{\gamma h}(w-\gamma s)$ in (7), 	
	$$0\in -\P_{\gamma}(w,s)+s+\partial h(\prox_{\gamma h}(w-\gamma s)).$$	
	Taking $p\in \partial h(\prox_{\gamma h}(w-\gamma s))$ such that
	$0=-\P_{\gamma}(w,s)+s+p$, it follows that 	
	\begin{alignat}{6}
	&0&&=\langle-\P_{\gamma}(w,s)+s+p,\P_{\gamma}(w,s)\rangle\nonumber\\
	&&&=\langle s+p,\P_{\gamma}(w,s)\rangle-||\P_{\gamma}(w,s)||^2_2\nonumber\\
	&&&\leq \langle s,\P_{\gamma}(w,s)\rangle +\frac{1}{\gamma}\left(h(w)-h(\prox_{\gamma 
		h}(w-\gamma 
	s))\right)-||\P_{\gamma}(w,s)||^2_2,\nonumber	
	\end{alignat}	
	where the inequality uses the convexity of $h$.	
\end{proof}
\begin{property}	
	\label{eq:none}
	\begin{alignat}{6}
	&\||\G^k_{\gamma, A}(w^k)-\G^k_{\gamma, E}(w^k)||_2\leq||\nabla A^k_{\lambda, 
		M}(w^k,\xi^k)-\nabla 
	E^k_{\lambda}(w^k)||_2\nonumber
	\end{alignat}	
\end{property}
\begin{proof}	
	\begin{alignat}{6}
	&\||\G^k_{\gamma, A}(w^k)-\G^k_{\gamma, E}(w^k)||_2&=&||\frac{1}{\gamma}\left(w^k-\prox_{\gamma 
		h}(w^k-\gamma 
	\nabla A^k_{\lambda, M}(w^k,\xi^k))\right)\nonumber\\
	&&&-\frac{1}{\gamma}\left(w^k-\prox_{\gamma h}(w^k-\gamma \nabla E^k_{\lambda		
	}(w^k))\right)||_2\nonumber\\
	&&=&\frac{1}{\gamma}||\prox_{\gamma h}(w^k-\gamma \nabla E^k_{\lambda 
	}(w^k))-\prox_{\gamma h}(w^k-\gamma \nabla A^k_{\lambda, M}(w^k,\xi^k))||_2\nonumber\\
	&&\leq&\frac{1}{\gamma}||w^k-\gamma\nabla E^k_{\lambda 
	}(w^k)-w^k+\gamma\nabla A^k_{\lambda, M}(w^k,\xi^k)||_2\nonumber\\
	&&=&||\nabla A^k_{\lambda, M}(w^k,\xi^k))- \nabla E^k_{\lambda}(w^k))||_2,\nonumber\\
	\end{alignat} 	
	where the inequality holds due to the nonexpansivity of the proximal operator of proper closed 
	convex functions \citep[Theorem 6.42]{beck2017}.
\end{proof}
\begin{proof}[Proof of Lemma 2]
	Given the smoothness of $E^k_{\lambda}(w)$ as shown in Property 1,
	\begin{alignat}{6}
	&E^k_{\lambda}(w^{k+1})&&\leq E^k_{\lambda}(w^k)+\langle\nabla 
	E^k_{\lambda}(w^k),w^{k+1}-w^k\rangle+\frac{L_{\lambda}}{2}||w^{k+1}-w^k||^2_2\nonumber\\
	&&&=E^k_{\lambda}(w^k)+\langle\nabla 
	E^k_{\lambda}(w^k),-\gamma\G^k_{\gamma, 
		A}(w^k)\rangle+\frac{L_{\lambda}}{2}||\gamma\G^k_{\gamma, A}(w^k)||^2_2,\nonumber
	\end{alignat}	
	from (17). Let $\delta_k:=\nabla A^k_{\lambda, M}(w^k,\xi^k)-\nabla 
	E^k_{\lambda}(w^k)$, then
	\begin{alignat}{6}
	&E^k_{\lambda}(w^{k+1})&&\leq E^k_{\lambda}(w^k)-\gamma\langle\nabla A^k_{\lambda, 		
	M}(w^k,\xi^k),\G^k_{\gamma, A}(w^k)\rangle+\gamma\langle \delta_k,\G^k_{\gamma, 
	A}(w^k)\rangle+\frac{L_{\lambda}}{2}||\gamma\G^k_{\gamma,A}(w^k)||^2_2.\nonumber
	\end{alignat}		
	Using Property \ref{eq:psize} with $w=w^k$ and $s=\nabla A^k_{\lambda, M}(w^k,\xi^k)$,
	\begin{alignat}{6}
	&E^k_{\lambda}(w^{k+1})&&\leq 
	E^k_{\lambda}(w^k)+h(w^k)-h(w^{k+1})-\gamma||\G^k_{\gamma, A}(w^k)||^2_2+\gamma\langle 
	\delta_k,\G^k_{\gamma, A}(w^k)\rangle+\frac{L_{\lambda}\gamma^2}{2}||\G^k_{\gamma, 
		A}(w^k)||^2_2.\nonumber
	\end{alignat}		
	Applying \eqref{eq:e1} and \eqref{eq:e3},	
	\begin{alignat}{6}
	&\tPhi_{\lambda}(w^{k+1})&&\leq 
	\tPhi_{\lambda}(w^k)-\gamma||\G^k_{\gamma, A}(w^k)||^2_2+\gamma\langle 
	\delta_k,\G^k_{\gamma, A}(w^k)\rangle+\frac{L_{\lambda}\gamma^2}{2}||\G^k_{\gamma, 
		A}(w^k)||^2_2\nonumber\\
	&&&= 
	\tPhi_{\lambda}(w^k)+\left(\frac{L_{\lambda}\gamma^2}{2}-\gamma\right)||\G^k_{\gamma, 
		A}(w^k)||^2_2+\gamma\langle 
	\delta_k,\G^k_{\gamma, E}(w^k)\rangle+\gamma\langle \delta_k,\G^k_{\gamma, 
		A}(w^k)-\G^k_{\gamma, 
		E}(w^k)\rangle\nonumber\\
	&&&\leq
	\tPhi_{\lambda}(w^k)+\left(\frac{L_{\lambda}\gamma^2}{2}-\gamma\right)||\G^k_{\gamma, 
		A}(w^k)||^2_2+\gamma\langle 
	\delta_k,\G^k_{\gamma, E}(w^k)\rangle+\gamma||\delta_k||_2||\G^k_{\gamma, A}(w^k)-\G^k_{\gamma, 
		E}(w^k)||_2\nonumber\\
	&&&\leq
	\tPhi_{\lambda}(w^k)+\left(\frac{L_{\lambda}\gamma^2}{2}-\gamma\right)||\G^k_{\gamma, 
		A}(w^k)||^2_2+\gamma\langle 
	\delta_k,\G^k_{\gamma, E}(w^k)\rangle+\gamma||\delta_k||^2_2,\nonumber
	\end{alignat}	
	where the last inequality uses Property \ref{eq:none}. After $N$ iterations, 	
	\begin{alignat}{6}
	&&\left(\gamma-\frac{L_{\lambda}\gamma^2}{2}\right)\sum_{k=1}^N ||\G^k_{\gamma, 
		A}(w^k)||^2_2&\leq 
	\tPhi_{\lambda}(w^1)-\tPhi_{\lambda}(w^{N+1})+\gamma\sum_{k=1}^N\left(\langle 
	\delta_k,\G^k_{\gamma, E}(w^k)\rangle+||\delta_k||^2_2\right)\nonumber\\
	&&&\leq \tPhi_{\lambda}(w^1)-\tPhi_{\lambda}(w^*)+\gamma\sum_{k=1}^N\left(\langle 
	\delta_k,\G^k_{\gamma, E}(w^k)\rangle+||\delta_k||^2_2\right)\nonumber
	\end{alignat}	
	It follows from (5) that for $w$ independent of $\xi^k$, $\EE[\nabla 
	A^k_{\lambda, M}(w,\xi^k)]=\nabla E^k_{\lambda}(w)$, and so $\EE[\delta_k]=0$. Taking the 
	expectation of both sides,	
	\begin{alignat}{6}
	&& \left(\gamma-\frac{L_{\lambda}}{2}\gamma^2\right)\sum_{k=1}^N\EE||\G^k_{\gamma, 
		A}(w^k)||^2_2\leq& 
	\tPhi_{\lambda}(w^1)-\tPhi_{\lambda}(w^*)+\gamma\sum_{k=1}^N\EE||\delta_k||^2_2\nonumber\\
	&&\leq& 
	\tPhi_{\lambda}(w^1)-\tPhi_{\lambda}(w^*)+\gamma\frac{N}{M}\sigma^2,\nonumber
	\end{alignat}	
	where the second inequality uses Property \ref{eq:minib}. As we choose $R$ uniformly over 
	$\{1,...,N\}$,	
	\begin{alignat}{6}
	&&\EE||\G^R_{\gamma, A}(w^R)||^2_2&=\frac{1}{N}\sum_{k=1}^N\EE||\G^k_{\gamma, 
		A}(w^k)||^2_2\nonumber\\
	&&&\leq\frac{1}{N\left(\gamma-\frac{L_{\lambda}}{2}\gamma^2\right)}\left(\tPhi_{\lambda}(w^1)-\tPhi_{\lambda}(w^*)+\gamma\frac{N}{M}\sigma^2\right)\nonumber\\
	&&&=\frac{2L_{\lambda}}{N}\left(\tPhi_{\lambda}(w^1)-\tPhi_{\lambda}(w^*)\right)+\frac{2}{M}\sigma^2,\nonumber
	\end{alignat}	
	where the final equality holds since $\gamma=\frac{1}{L_{\lambda}}$.	
	\begin{alignat}{6}
	&&\EE||\G^R_{\gamma, E}(w^R)||^2_2&=\EE||\G^R_{\gamma, A}(w^R)+\G^R_{\gamma, 
		E}(w^R)-\G^R_{\gamma, A}(w^R)||^2_2 \nonumber\\
	&&&\leq \EE\left(||\G^R_{\gamma, A}(w^R)||^2_2+2\langle 
	\G^R_{\gamma, A}(w^R),\G^R_{\gamma,E}(w^R)-\G^R_{\gamma, A}(w^R)\rangle+||\G^R_{\gamma, 
		E}(w^R)-\G^R_{\gamma, A}(w^R)||^2_2\right) \nonumber\\
	&&&\leq2\EE||\G^R_{\gamma, A}(w^R)||^2_2+2\EE||\G^R_{\gamma, 
		E}(w^R)-\G^R_{\gamma, A}(w^R)||^2_2 \nonumber\\
	&&&\leq 
	\frac{4L_{\lambda}}{N}\left(\tPhi_{\lambda}(w^1)-\tPhi_{\lambda}(w^*)\right)+\frac{4}{M}\sigma^2+2\EE||\nabla
	A^R_{\lambda, M}(w^R,\xi^R)-\nabla E^R_{\lambda}(w^R)||^2_2\nonumber\\
	&&&\leq 
	\frac{4L_{\lambda}}{N}\left(\tPhi_{\lambda}(w^1)-\tPhi_{\lambda}(w^*)\right)+\frac{4}{M}\sigma^2+\frac{2}{M}\sigma^2\nonumber\\
	&&&= 
	\frac{4L_{\lambda}}{N}\left(\tPhi_{\lambda}(w^1)-\tPhi_{\lambda}(w^*)\right)+\frac{6}{M}\sigma^2\nonumber\\
	&&&=\frac{(L+N^{\theta})}{N}\tilde{\DD}+\frac{6}{\lceil
		N^{\alpha}\rceil}\sigma^2,\nonumber
	\end{alignat}	
	where the second inequality uses Young's inequality on the middle term.	
\end{proof}

\section{Proofs of Properties 3 and 4}
\begin{customproperty}{3}	
	Assume that $g(w)$ is Lipschitz continuous with parameter $l$ and $\bar{\gamma}\geq \gamma$, then
	\begin{alignat}{6}
	&&\dist(0,\G_{\bar{\gamma}}(\zeta^{\lambda}(w^k)))&\leq 	
	&||\G^k_{\gamma, E}(w^k)||_2+2l\lambda\left(\frac{2}{\bar{\gamma}}+L\right).\nonumber
	\end{alignat}	
\end{customproperty}
In order to prove Property 3 we require the following property. Its proof can be found 
in \citep[Theorem 10.9]{beck2017} under a slightly different setting, so we present it here for 
clarity. 
\begin{property}
	\label{eq:mono}	
	For $\gamma^1\geq \gamma^2>0$ and any $w,s\in \RR^d$,
	$$||\P_{\gamma^1}(w,s)||_2\leq ||\P_{\gamma^2}(w,s)||_2.$$ 	
\end{property}
\begin{proof}	
	For an arbitrary $v\in \RR^d$ and $\gamma>0$, $\prox_{\gamma h}(v)$ is the minimizer of 
	$\frac{1}{2\gamma}||v-x||^2_2+h(x)$ and so   
	\begin{alignat}{6}
	&\frac{1}{\gamma}(v-\prox_{\gamma h}(v))\in \partial h(\prox_{\gamma h}(v)).\label{eq:subg}
	\end{alignat}	
	By the definition of a subgradient of a convex function, for any $y\in \RR^d$,	
	$$h(y)-h(\prox_{\gamma h}(v))\geq\frac{1}{\gamma}\left<v-\prox_{\gamma h}(v),y-\prox_{\gamma 
		h}(v)\right>.$$	
	First let $\gamma=\gamma^1$, $v=w-\gamma^1 s$ and $y=\prox_{\gamma^2h}(w-\gamma^2s)$,
	\begin{alignat}{6}
	&&&h(\prox_{\gamma^2h}(w-\gamma^2s))-h(\prox_{\gamma^1 h}(w-\gamma^1s))\label{eq:one}\\
	&&\geq&\frac{1}{\gamma^1}\left<w-\gamma^1 s-\prox_{\gamma^1 h}(w-\gamma^1 
	s),\prox_{\gamma^2h}(w-\gamma^2s)-\prox_{\gamma^1 
		h}(w-\gamma^1 s)\right>\nonumber\\
	&&=&\left<\P_{\gamma^1}(w,s)- 
	s,\gamma^1\P_{\gamma^1}(w,s)-\gamma^2\P_{\gamma^2}(w,s)\right>.\nonumber
	\end{alignat}	
	Exchanging $\gamma^1$ and $\gamma^2$, letting $\gamma=\gamma^2$, $v=w-\gamma^2 s$ and 
	$y=\prox_{\gamma^1h}(w-\gamma^1s)$,
	\begin{alignat}{6}
	&h(\prox_{\gamma^1h}(w-\gamma^1s))-h(\prox_{\gamma^2 
		h}(w-\gamma^2s))\geq\left<\P_{\gamma^2}(w,s)-s,\gamma^2\P_{\gamma^2}(w,s)-\gamma^1\P_{\gamma^1}(w,s)\right>.\label{eq:two}
	\end{alignat}	
	Adding inequalities \eqref{eq:one} and \eqref{eq:two},
	\begin{alignat}{6}
	&0&\geq&\left<\P_{\gamma^1}(w,s)- 
	\P_{\gamma^2}(w,s),\gamma^1\P_{\gamma^1}(w,s)-\gamma^2\P_{\gamma^2}(w,s)\right>.\label{eq:thr}
	\end{alignat}	
	Expanding and rearranging \eqref{eq:thr},
	\begin{alignat}{6}
	&\gamma^1||\P_{\gamma^1}(w,s)||^2_2+\gamma^2||\P_{\gamma^2}(w,s)||^2_2&\leq&
	(\gamma^1+\gamma^2)\left<\P_{\gamma^1}(w,s),\P_{\gamma^2}(w,s)\right>\label{eq:csin}\\
	&&\leq&(\gamma^1+\gamma^2)||\P_{\gamma^1}(w,s)||_2||\P_{\gamma^2}(w,s)||_2,\nonumber
	\end{alignat}	
	using the Cauchy-Schwarz inequality. Assume $||\P_{\gamma^1}(w,s)||_2>0$, otherwise the 
	property trivially holds, and set $t=||\P_{\gamma^2}(w,s)||_2/||\P_{\gamma^1}(w,s)||_2$. 
	Inequality \eqref{eq:csin} can now be written as 
	\begin{alignat}{6}
	&\gamma^1+\gamma^2t^2-(\gamma^1+\gamma^2)t&\leq&0.\nonumber
	\end{alignat}	
	The roots of the left hand side function occur at $t=1$ and $t=\frac{\gamma^1}{\gamma^2}$, so 
	for the inequality to hold, 
	$$1\leq t\leq \frac{\gamma^1}{\gamma^2},$$
	which includes the desired inequality,
	$$||\P_{\gamma^1}(w,s)||_2\leq ||\P_{\gamma^2}(w,s)||_2.$$		
\end{proof}
\begin{proof}[Proof of Property 3]
	Given that $\zeta^{\lambda}(w)$ is a minimizer of $\frac{1}{2\lambda}||w-x||^2_2+g(x)$ from 
	(8),    
	\begin{alignat}{6}
	&\frac{1}{\lambda}(w-\zeta^{\lambda}(w))\in \partial g(\zeta^{\lambda}(w)).
	\end{alignat}	
	It follows that
	\begin{alignat}{6}
	&\dist(0,\G_{\bar{\gamma}}(\zeta^{\lambda}(w^k)))&\leq& 
	||\P_{\bar{\gamma}}(\zeta^{\lambda}(w^k),\nabla 
	f(\zeta^{\lambda}(w^k))+\frac{1}{\lambda}(w^k-\zeta^{\lambda}(w^k)))||_2\nonumber\\
	&&=&||\P_{\bar{\gamma}}(\zeta^{\lambda}(w^k),\nabla 
	f(\zeta^{\lambda}(w^k))+\frac{1}{\lambda}(w^k-\zeta^{\lambda}(w^k)))\nonumber\\
	&&&+\P_{\bar{\gamma}}(w^k,\nabla 
	E^k_{\lambda}(w^k))-\P_{\bar{\gamma}}(w^k,\nabla E^k_{\lambda}(w^k))||_2\nonumber\\
	&&\leq&||\P_{\bar{\gamma}}(w^k,\nabla 
	E^k_{\lambda}(w^k))||_2\nonumber\\
	&&&+||\P_{\bar{\gamma}}(\zeta^{\lambda}(w^k),\nabla 
	f(\zeta^{\lambda}(w^k))+\frac{1}{\lambda}(w^k-\zeta^{\lambda}(w^k)))-\P_{\bar{\gamma}}(w^k,\nabla 
	E^k_{\lambda}(w^k))||_2.\nonumber
	\end{alignat}	
	Given that $||\P_{\bar{\gamma}}(w^k,\nabla 
	E^k_{\lambda}(w^k))||_2\leq ||\P_{\gamma}(w^k,\nabla 
	E^k_{\lambda}(w^k))||_2=||\G^k_{\gamma, E}(w^k)||_2$ from Property \ref{eq:mono},
	\begin{alignat}{6}
	&\dist(0,\G_{\bar{\gamma}}(\zeta^{\lambda}(w^k)))&\leq&||\G^k_{\gamma, E}(w^k)||_2\nonumber\\
	&&&+||\P_{\bar{\gamma}}(\zeta^{\lambda}(w^k),\nabla 
	f(\zeta^{\lambda}(w^k))+\frac{1}{\lambda}(w^k-\zeta^{\lambda}(w^k)))-\P_{\bar{\gamma}}(w^k,\nabla 
	E^k_{\lambda}(w^k))||_2.\label{comb1}
	\end{alignat}
	Focusing on the second term, 
	\begin{alignat}{6}
	&||\P_{\bar{\gamma}}(\zeta^{\lambda}(w^k),\nabla 
	f(\zeta^{\lambda}(w^k))+\frac{1}{\lambda}(w^k-\zeta^{\lambda}(w^k)))-\P_{\bar{\gamma}}(w^k,\nabla
	E^k_{\lambda}(w^k))||_2\nonumber\\
	=&\left\Vert\frac{1}{\bar{\gamma}}\left(\zeta^{\lambda}(w^k)-\prox_{\bar{\gamma} 		
	h}(\zeta^{\lambda}(w^k)-\bar{\gamma}(\nabla 
	f(\zeta^{\lambda}(w^k))+\frac{1}{\lambda}(w^k-\zeta^{\lambda}(w^k))))\right)\right.\nonumber\\
	&\left.-\frac{1}{\bar{\gamma}}\left(w^k-\prox_{\bar{\gamma}
		h}(w^k-\bar{\gamma} (\nabla 
	f(w^k)+\frac{1}{\lambda}(w^k-\zeta^{\lambda}(w^k))))\right)\right\Vert_2\nonumber\\
	\leq&\frac{1}{\bar{\gamma}}||\zeta^{\lambda}(w^k)-w^k||_2+\frac{1}{\bar{\gamma}}||\prox_{\bar{\gamma}h}
	(w^k-\bar{\gamma} 
	(\nabla f(w^k)+\frac{1}{\lambda}(w^k-\zeta^{\lambda}(w^k))))\nonumber\\
	&-\prox_{\bar{\gamma} 
		h}(\zeta^{\lambda}(w^k)-\bar{\gamma} 
	(\nabla f(\zeta^{\lambda}(w^k))+\frac{1}{\lambda}(w^k-\zeta^{\lambda}(w^k))))||_2\nonumber\\
	\leq&\frac{1}{\bar{\gamma}}||\zeta^{\lambda}(w^k)-w^k||_2+\frac{1}{\bar{\gamma}}||w^k-\bar{\gamma} 
	(\nabla f(w^k)+\frac{1}{\lambda}(w^k-\zeta^{\lambda}(w^k)))\nonumber\\
	&-(\zeta^{\lambda}(w^k)-\bar{\gamma} 
	(\nabla f(\zeta^{\lambda}(w^k))+\frac{1}{\lambda}(w^k-\zeta^{\lambda}(w^k))))||_2\nonumber\\
	\leq&\frac{2}{\bar{\gamma}}||\zeta^{\lambda}(w^k)-w^k||_2+||
	\nabla f(\zeta^{\lambda}(w^k))-\nabla f(w^k)||_2\nonumber\\
	\leq&\frac{2}{\bar{\gamma}}||\zeta^{\lambda}(w^k)-w^k||_2+L||\zeta^{\lambda}(w^k)-w^k||_2,\label{comb2}
	\end{alignat}	
	where the second inequality follows from the nonexpansivity of the proximal operator. In order to 
	bound $||\zeta^{\lambda}(w^k)-w^k||_2$, recall from (12) that    				  
	\begin{alignat}{6}
	&&g(w)&\geq e_{\lambda}g(w)\nonumber\\
	&&&=\frac{1}{2\lambda}||w-\zeta^{\lambda}(w)||_2^2+g(\zeta^{\lambda}(w)).\nonumber
	\end{alignat}
	Rearranging and using the Lipschitz continuity of $g(w)$,
	\begin{alignat}{6}
	&&\frac{1}{2\lambda}||w-\zeta^{\lambda}(w)||_2^2&\leq 
	g(w)-g(\zeta^{\lambda}(w))\nonumber\\
	&&&\leq l||w-\zeta^{\lambda}(w)||_2\nonumber\\
	&&||w-\zeta^{\lambda}(w)||_2&\leq 2l\lambda.\label{comb3}
	\end{alignat}	
	Using \eqref{comb1}-\eqref{comb3},
	\begin{alignat}{6}
	&\dist(0,\G_{\bar{\gamma}}(\zeta^{\lambda}(w^k)))&\leq&||\G^k_{\gamma, 
		E}(w^k)||_2+\frac{2}{\bar{\gamma}}||\zeta^{\lambda}(w^k)-w^k||_2+L||\zeta^{\lambda}(w^k)-w^k||_2\nonumber\\
	&&\leq&||\G^k_{\gamma, E}(w^k)||_2+2l\lambda\left(\frac{2}{\bar{\gamma}}+L\right).\nonumber
	\end{alignat}	
\end{proof}
\begin{customproperty}{4}	
	Let $w^*$ be a global minimizer of $\Phi(\cdot)$ and let $w^*_{\lambda}$ be a global minimizer 
	of $\tPhi_{\lambda}(\cdot)$. Assume that $g(w)$ is Lipschitz continuous with parameter $l$, then  
	$$\tPhi_{\lambda}(w)-\tPhi_{\lambda}(w^*_{\lambda})\leq 
	\Phi(w)-\Phi(w^*) + \frac{l^2\lambda}{2}.$$	
\end{customproperty}
\begin{proof}	
	\begin{alignat}{6}
	&\tPhi_{\lambda}(w)-\tPhi_{\lambda}(w^*_{\lambda})-\Phi(w)+\Phi(w^*)&=&e_{\lambda}g(w)-
	f(w_{\lambda}^*) -e_{\lambda}g(w_{\lambda}^*)-h(w_{\lambda}^*)\nonumber\\
	&&&-g(w)+f(w^*)+g(w^*)+h(w^*)\nonumber\\
	&&\leq&- f(w_{\lambda}^*) -e_{\lambda}g(w_{\lambda}^*)-h(w_{\lambda}^*)+f(w^*) 
	+g(w^*)+h(w^*)\nonumber\\
	&&\leq&- f(w_{\lambda}^*) -e_{\lambda}g(w_{\lambda}^*)-h(w_{\lambda}^*)+f(w_{\lambda}^*) 
	+g(w_{\lambda}^*)+h(w_{\lambda}^*)\nonumber\\
	&&=&g(w_{\lambda}^*)-e_{\lambda}g(w_{\lambda}^*),\nonumber
	\end{alignat}	
	where the first inequality follows from (12). For any $w$, 
	\begin{alignat}{6}
	&&e_{\lambda}g(w)&=\frac{1}{2\lambda}||w-\zeta^{\lambda}(w)||_2^2+g(\zeta^{\lambda}(w))\nonumber\\
	&&g(w)-e_{\lambda}g(w)&=g(w)-g(\zeta^{\lambda}(w))-\frac{1}{2\lambda}||w-\zeta^{\lambda}(w)||_2^2\nonumber\\
	&&&\leq l||w-\zeta^{\lambda}(w)||_2-\frac{1}{2\lambda}||w-\zeta^{\lambda}(w)||_2^2.\nonumber
	\end{alignat}	
	The right-hand side is maximized when $||w-\zeta^{\lambda}(w)||_2=l\lambda$, giving the 
	desired result, 		
	\begin{alignat}{6}\label{eq:propK1} &g(w)-e_{\lambda}g(w)\leq \frac{l^2\lambda}{2}.
	\end{alignat}	
\end{proof}

\section{Proof of Lemma 7}
\begin{customlemma}{7}
	For an initial value $\tilde{w}_1$, $N\in\ZZ_{>0}$, and $\alpha,\theta\in\RR$, VRSGA 
	generates $w^R_T$ satisfying the following bound.	  	  	
	\begin{alignat}{6} 
	&&\EE\left[||\G^{R,T}_{\gamma, 
		E}(w^R_T)||^2_2\right]&\leq\tilde{\DD}\frac{L+(Sm)^{\theta}}{Sm}\nonumber
	\end{alignat}	
	where $\tilde{\DD}=36(\tPhi_{\lambda}(\tilde{w}^1)-\tPhi_{\lambda}(w^*_{\lambda}))$ and 
	$w^*_{\lambda}$ is a global minimizer of $\tPhi_{\lambda}(\cdot)$.	
\end{customlemma}
In order to prove this result, we require the following properties. The proof of Property 
\ref{vrprop1} can be found in \citep{li2018}, which we include here in our notation for clarity.
\begin{property}	
	\label{vrprop1}	
	Consider arbitrary $w,s,z\in \RR^d$, and $w^+=\prox_{\gamma h}(w-\gamma s)$,  	
	\begin{alignat}{6} 
	&E^{k,t}_{\lambda}(w^+)+h(w^+)&\leq& E^{k,t}_{\lambda}(z)+h(z)+\langle\nabla 
	E^{k,t}_{\lambda}(w)-s,w^+-z\rangle+\frac{L_{\lambda}}{2}||w^+-w||^2_2+\frac{L_{\lambda}}{2}||z-w||^2_2\nonumber\\
	&&&-\frac{1}{\gamma}\langle w^+-w,w^+-z\rangle.\nonumber	
	\end{alignat}	
\end{property}
\begin{proof}	
	As was done in the proof of Property \ref{eq:psize}, let us take $p\in \partial h(\prox_{\gamma 
	h}(w-\gamma s))$ such that $0=-\P_{\gamma}(w,s)+s+p=\frac{1}{\gamma}(w^+-w)+s+p$. It follows 
	by 	the convexity of $h(\cdot)$ that 	
	\begin{alignat}{6}\label{conin}
	&h(w^+)&&\leq h(z)+\langle p,w^+-z\rangle\nonumber\\
	&&&=h(z)-\bigg\langle\frac{1}{\gamma}(w^+-w)+s,w^+-z\bigg\rangle.
	\end{alignat}	
	Adding \eqref{conin} with the following two inequalities, which come from the smoothness of 
	$E^{k,t}_{\lambda}(w)$ and $-E^{k,t}_{\lambda}(w)$, see 
	Property 1, proves the result. 
	\begin{alignat}{6} &&E^{k,t}_{\lambda}(w^+)&\leq E^{k,t}_{\lambda}(w)+\langle\nabla 
	E^{k,t}_{\lambda}(w),w^+-w\rangle+\frac{L_{\lambda}}{2}||w^+-w||^2_2\nonumber\\
	&&-E^{k,t}_{\lambda}(z)&\leq -E^{k,t}_{\lambda}(w)+\langle-\nabla 
	E^{k,t}_{\lambda}(w),z-w\rangle+\frac{L_{\lambda}}{2}||z-w||^2_2\nonumber
	\end{alignat}	
\end{proof}
\begin{property}	
	\label{yineq}
	For vectors $w$, $x$, $z$, and $\beta>0$, 	
	\begin{alignat}{6} 
	&&||w-x||^2_2&\leq 
	(1+\beta)||w-z||^2_2+\left(1+\frac{1}{\beta}\right)||z-x||^2_2.\nonumber
	\end{alignat} 		
\end{property}
\begin{proof}	
	\begin{alignat}{6} 
	&&||w-x||^2_2&=||w-z+z-x||^2_2\nonumber\\
	&&&\leq \left(||w-z||_2+||z-x||_2\right)^2\nonumber\\
	&&&=||w-z||^2_2+2||w-z||_2||z-x||_2+||z-x||^2_2\nonumber\\
	&&&\leq||w-z||^2_2+\left(\beta||w-z||^2_2+\frac{1}{\beta}||z-x||^2_2\right)+||z-x||^2_2\nonumber\\
	&&&=(1+\beta)||w-z||^2_2+\left(1+\frac{1}{\beta}\right)||z-x||^2_2,\nonumber
	\end{alignat} 	
	where the second inequality uses Young's inequality.		
\end{proof}
\begin{proof}[Proof of Lemma 7]	
	Let $\hat{w}^k_{t+1}=\prox_{\gamma h}(w^k_t-\gamma \nabla E^{k,t}_{\lambda}(w^k_t))$, with 
	$w^+=w^k_{t+1}$, $w=w^k_t$, $s=V^k_t$, and $z=\hat{w}^k_{t+1}$ in Property \ref{vrprop1} to get 
	the inequality
	\begin{alignat}{6} 
	&E^{k,t}_{\lambda}(w^k_{t+1})+h(w^k_{t+1})&\leq& 
	E^{k,t}_{\lambda}(\hat{w}^k_{t+1})+h(\hat{w}^k_{t+1})+\langle\nabla 	
	E^{k,t}_{\lambda}(w^k_t)-V^k_t,w^k_{t+1}-\hat{w}^k_{t+1}\rangle+\nonumber\\
	&&&\frac{L_{\lambda}}{2}||w^k_{t+1}-w^k_t||^2_2+\frac{L_{\lambda}}{2}||\hat{w}^k_{t+1}-w^k_t||^2_2-\frac{1}{\gamma}\langle
	w^k_{t+1}-w^k_t,w^k_{t+1}-\hat{w}^k_{t+1}\rangle.\label{vr:ineq2}
	\end{alignat}	
	In addition, let $w^+=\hat{w}^k_{t+1}$, $w=w^k_t$, $s=\nabla E^{k,t}_{\lambda}(w^k_t)$, and 
	$z=w^k_t$ in Property \ref{vrprop1} to get
	\begin{alignat}{6} 
	&E^{k,t}_{\lambda}(\hat{w}^k_{t+1})+h(\hat{w}^k_{t+1})&\leq& 
	E^{k,t}_{\lambda}(w^k_t)+h(w^k_t)+\langle\nabla 
	E^{k,t}_{\lambda}(w^k_t)-\nabla 
	E^{k,t}_{\lambda}(w^k_t),\hat{w}^k_{t+1}-w^k_{t+1}\rangle+\frac{L_{\lambda}}{2}||\hat{w}^k_{t+1}-w^k_t||^2_2\nonumber\\
	&&&+\frac{L_{\lambda}}{2}||w^k_t-w^k_t||^2_2-\frac{1}{\gamma}\langle 
	\hat{w}^k_{t+1}-w^k_t,\hat{w}^k_{t+1}-w^k_t\rangle\nonumber\\	
	&&=&E^{k,t}_{\lambda}(w^k_t)+h(w^k_t)+\left(\frac{L_{\lambda}}{2}-\frac{1}{\gamma}\right)||\hat{w}^k_{t+1}-w^k_t||^2_2.\label{vr:ineq3}
	\end{alignat}	
	Adding \eqref{vr:ineq2} and \eqref{vr:ineq3}, 
	\begin{alignat}{6} 
	&E^{k,t}_{\lambda}(w^k_{t+1})+h(w^k_{t+1})&\leq& 
	E^{k,t}_{\lambda}(w^k_t)+h(w^k_t)+\langle\nabla 	
	E^{k,t}_{\lambda}(w^k_t)-V^k_t,w^k_{t+1}-\hat{w}^k_{t+1}\rangle+\frac{L_{\lambda}}{2}||w^k_{t+1}-w^k_t||^2_2\nonumber\\
	&&&-\frac{1}{\gamma}\langle 
	w^k_{t+1}-w^k_t,w^k_{t+1}-\hat{w}^k_{t+1}\rangle+	
	\left(L_{\lambda}-\frac{1}{\gamma}\right)||\hat{w}^k_{t+1}-w^k_t||^2_2.\nonumber
	\end{alignat}	
	From \eqref{eq:e1} and \eqref{eq:e3},	
	\begin{alignat}{6} 
	&\tPhi_{\lambda}(w^k_{t+1})&\leq& \tPhi_{\lambda}(w^k_t)+\langle\nabla 	
	E^{k,t}_{\lambda}(w^k_t)-V^k_t,w^k_{t+1}-\hat{w}^k_{t+1}\rangle+\frac{L_{\lambda}}{2}||w^k_{t+1}-w^k_t||^2_2\nonumber\\
	&&&-\frac{1}{\gamma}\langle 
	w^k_{t+1}-w^k_t,w^k_{t+1}-\hat{w}^k_{t+1}\rangle+	
	\left(L_{\lambda}-\frac{1}{\gamma}\right)||\hat{w}^k_{t+1}-w^k_t||^2_2.\label{keyineq}
	\end{alignat}	
	Plugging $\langle w^k_{t+1}-w^k_t,w^k_{t+1}-\hat{w}^k_{t+1}\rangle=\frac{1}{2}\left(|| 
	w^k_{t+1}-w^k_t||^2_2+||w^k_{t+1}-\hat{w}^k_{t+1}||^2_2-||\hat{w}^k_{t+1}-w^k_t||^2_2\right)$ into 
	\eqref{keyineq} and rearranging,	
	\begin{alignat}{6} 
	&\tPhi_{\lambda}(w^k_{t+1})&\leq& \tPhi_{\lambda}(w^k_t)+\langle\nabla 	
	E^{k,t}_{\lambda}(w^k_t)-V^k_t,w^k_{t+1}-\hat{w}^k_{t+1}\rangle+\left(\frac{L_{\lambda}}{2}-\frac{1}{2\gamma}\right)||w^k_{t+1}-w^k_t||^2_2\nonumber\\
	&&&-\frac{1}{2\gamma}||w^k_{t+1}-\hat{w}^k_{t+1}||^2_2+	
	\left(L_{\lambda}-\frac{1}{2\gamma}\right)||\hat{w}^k_{t+1}-w^k_t||^2_2.\label{byi}
	\end{alignat} 	
	Focusing on the term $-\frac{1}{2\gamma}||w^k_{t+1}-\hat{w}^k_{t+1}||^2_2$, we apply Property 
	\ref{yineq} with $w=w^k_{t+1}$, $x=w^k_t$, and $z=\hat{w}^k_{t+1}$. After rearranging, 	
	\begin{alignat}{6} 
	&&-(1+\beta)||w^k_{t+1}-\hat{w}^k_{t+1}||^2_2&\leq 
	-||w^k_{t+1}-w^k_t||^2_2+\left(1+\frac{1}{\beta}\right)||\hat{w}^k_{t+1}-w^k_t||^2_2\nonumber\\
	&&-\frac{1}{2\gamma}||w^k_{t+1}-\hat{w}^k_{t+1}||^2_2&\leq 
	-\frac{1}{(1+\beta)2\gamma} 	
	||w^k_{t+1}-w^k_t||^2_2+\frac{\left(1+\frac{1}{\beta}\right)}{(1+\beta)2\gamma}||\hat{w}^k_{t+1}-w^k_t||^2_2.\nonumber
	\end{alignat} 
	Choosing $\beta=3$,
	\begin{alignat}{6} 
	&&-\frac{1}{2\gamma}||w^k_{t+1}-\hat{w}^k_{t+1}||^2_2&\leq 
	-\frac{1}{8\gamma} 
	||w^k_{t+1}-w^k_t||^2_2+\frac{1}{6\gamma}||\hat{w}^k_{t+1}-w^k_t||^2_2.\nonumber
	\end{alignat} 	
	Using this inequality in \eqref{byi},
	\begin{alignat}{6} 
	&\tPhi_{\lambda}(w^k_{t+1})&\leq& \tPhi_{\lambda}(w^k_t)+\langle\nabla 
	E^{k,t}_{\lambda}(w^k_t)-V^k_t,w^k_{t+1}-\hat{w}^k_{t+1}\rangle+\left(\frac{L_{\lambda}}{2}-\frac{1}{2\gamma}\right)||w^k_{t+1}-w^k_t||^2_2\nonumber\\
	&&&-\frac{1}{8\gamma} ||w^k_{t+1}-w^k_t||^2_2+\frac{1}{6\gamma}||\hat{w}^k_{t+1}-w^k_t||^2_2+	
	\left(L_{\lambda}-\frac{1}{2\gamma}\right)||\hat{w}^k_{t+1}-w^k_t||^2_2\nonumber\\
	&&=& \tPhi_{\lambda}(w^k_t)+\langle\nabla 
	E^{k,t}_{\lambda}(w^k_t)-V^k_t,w^k_{t+1}-\hat{w}^k_{t+1}\rangle+\left(\frac{L_{\lambda}}{2}-\frac{5}{8\gamma}\right)||w^k_{t+1}-w^k_t||^2_2\nonumber\\
	&&&+\left(L_{\lambda}-\frac{1}{3\gamma}\right)||\hat{w}^k_{t+1}-w^k_t||^2_2\nonumber\\
	&&\leq& \tPhi_{\lambda}(w^k_t)+\gamma ||\nabla 
	E^{k,t}_{\lambda}(w^k_t)-V^k_t||^2_2+\left(\frac{L_{\lambda}}{2}-\frac{5}{8\gamma}\right)||w^k_{t+1}-w^k_t||^2_2\nonumber\\
	&&&+\left(L_{\lambda}-\frac{1}{3\gamma}\right)||\hat{w}^k_{t+1}-w^k_t||^2_2,\label{nonexp}		
	\end{alignat}	
	where the last inequality holds since	
	\begin{alignat}{6} 
	&\langle\nabla 
	E^{k,t}_{\lambda}(w^k_t)-V^k_t,w^k_{t+1}-\hat{w}^k_{t+1}\rangle&\leq& ||\nabla 
	E^{k,t}_{\lambda}(w^k_t)-V^k_t||_2||w^k_{t+1}-\hat{w}^k_{t+1}||_2\nonumber\\
	&&=& ||\nabla 
	E^{k,t}_{\lambda}(w^k_t)-V^k_t||_2||\prox_{\gamma h}(w^k_t-\gamma 
	V^k_t)-\prox_{\gamma h}(w^k_t-\gamma \nabla E^{k,t}_{\lambda}(w^k_t))||_2\nonumber\\	
	&&\leq& \gamma||\nabla 
	E^{k,t}_{\lambda}(w^k_t)-V^k_t||^2_2\nonumber
	\end{alignat}	
	using the Cauchy-Schwarz inequality and the nonexpansivity of the proximal operator of $h$. Taking 
	the expectation of both sides of \eqref{nonexp},	
	\begin{alignat}{6} 
	&\EE\tPhi_{\lambda}(w^k_{t+1})&\leq& \EE\left[\tPhi_{\lambda}(w^k_t)+\gamma 
	||\nabla E^{k,t}_{\lambda}(w^k_t)-V^k_t||^2_2
	+\left(\frac{L_{\lambda}}{2}-\frac{5}{8\gamma}\right)||w^k_{t+1}-w^k_t||^2_2\right.\nonumber\\	
	&&&\left.+\left(L_{\lambda}-\frac{1}{3\gamma}\right)||\hat{w}^k_{t+1}-w^k_t||^2_2\right].\label{bvb}
	\end{alignat}	
	Focusing on $\EE\left[||\nabla E^{k,t}_{\lambda}(w^k_t)-V^k_t||^2_2\right]$,  	
	from (15) and the definition of $V^k_t$ found in Algorithm 2, $\nabla 
	E^{k,t}_{\lambda}(w^k_t)-V^k_t=\nabla f(w)-(\frac{1}{b}\sum_{j\in 
		I}\left(\nabla f_j(w^k_t)-\nabla f_j(\tilde{w}^k)\right)+G^k)$. After rearranging,
	\begin{alignat}{6} 
	&&\EE||\nabla E^{k,t}_{\lambda}(w^k_t)-V^k_t||^2_2&=\EE||\frac{1}{b}\sum_{j\in I}\left(\nabla 
	f_j(\tilde{w}^k)-\nabla f_j(w^k_t)\right)-\left(G^k-\nabla f(w^k_t)\right)||^2_2\nonumber\\
	&&&=\frac{1}{b^2}\EE\sum_{j\in I}||\nabla f_j(\tilde{w}^k)-\nabla f_j(w^k_t)-\left(G^k-\nabla 
	f(w^k_t)\right)||^2_2\nonumber\\
	&&&\leq\frac{1}{b^2}\EE\sum_{j\in I}||\nabla f_j(\tilde{w}^k)-\nabla 
	f_j(w^k_t)||^2_2\nonumber\\
	&&&\leq\frac{L^2}{b}\EE||\tilde{w}^k-w^k_t||^2_2.\nonumber
	\end{alignat}	
	As the expectation of the squared norm of a sum of independent random variables with zero mean, 
	the second equality holds using the same reasoning found in Property \ref{eq:minib}, and the first 
	inequality holds since $\EE||x-\EE[x]||^2_2\leq \EE||x||^2_2$ for any random variable $x$. Using 
	this bound in \eqref{bvb},	
	\begin{alignat}{6} 
	&\EE\tPhi_{\lambda}(w^k_{t+1})&\leq& \EE\left[\tPhi_{\lambda}(w^k_t)+\gamma 
	\frac{L^2}{b}||\tilde{w}^k-w^k_t||^2_2+
	\left(\frac{L_{\lambda}}{2}-\frac{5}{8\gamma}\right)||w^k_{t+1}-w^k_t||^2_2\right.\nonumber\\
	&&&\left.+\left(L_{\lambda}-\frac{1}{3\gamma}\right)||\hat{w}^k_{t+1}-w^k_t||^2_2\right]\nonumber\\
	&&\leq&\EE\left[\tPhi_{\lambda}(w^k_t)+\frac{L_{\lambda}}{6b}||\tilde{w}^k-w^k_t||^2_2-
	\frac{13L_{\lambda}}{4}||w^k_{t+1}-w^k_t||^2_2-L_{\lambda}||\hat{w}^k_{t+1}-w^k_t||^2_2\right]\nonumber\\
	&&=&\EE\left[\tPhi_{\lambda}(w^k_t)+\frac{L_{\lambda}}{6b}||\tilde{w}^k-w^k_t||^2_2-
	\frac{13L_{\lambda}}{4}||w^k_{t+1}-w^k_t||^2_2-\frac{1}{36L_{\lambda}}||\G^{k,t}_{\gamma 
	E}(w^k_t)||^2_2\right],\label{byi2}
	\end{alignat}		
	where the last two lines use $\gamma=\frac{1}{6L_{\lambda}}$. Focusing on 
	$-\frac{13L_{\lambda}}{4}||w^k_{t+1}-w^k_t||^2_2$, we apply Property \ref{yineq} with 
	$w=w^k_{t+1}$, $x=\tilde{w}^k$, and $z=w^k_{t}$, 
	\begin{alignat}{6} 
	&&(1+\beta)||w^k_{t+1}-w^k_{t}||^2_2&\geq  
	||w^k_{t+1}-\tilde{w}^k||^2_2-\left(1+\frac{1}{\beta}\right)||w^k_{t}-\tilde{w}^k||^2_2\nonumber\\
	&&-\frac{13L_{\lambda}}{4}||w^k_{t+1}-w^k_{t}||^2_2&\leq  
	-\frac{13L_{\lambda}}{4(1+\beta)}||w^k_{t+1}-\tilde{w}^k||^2_2+\frac{13L_{\lambda}\left(1+\frac{1}{\beta}\right)}{4(1+\beta)}||w^k_{t}-\tilde{w}^k||^2_2.\nonumber
	\end{alignat} 		
	Setting $\beta=2t-1$, 	
	$$-\frac{13L_{\lambda}}{4}||w^k_{t+1}-w^k_{t}||^2_2\leq 
	-\frac{13L_{\lambda}}{8t}||w^k_{t+1}-\tilde{w}^k||^2_2+\frac{13L_{\lambda}}{8t-4}||w^k_{t}-\tilde{w}^k||^2_2.$$
	Applying this bound in \eqref{byi2},	
	\begin{alignat}{6} 
	&\EE\tPhi_{\lambda}(w^k_{t+1})&\leq& 	
	\EE\left[\tPhi_{\lambda}(w^k_t)+\left(\frac{L_{\lambda}}{6b}+\frac{13L_{\lambda}}{8t-4}\right)||\tilde{w}^k-w^k_t||^2_2-
	\frac{13L_{\lambda}}{8t}||w^k_{t+1}-\tilde{w}^k||^2_2\right.\nonumber\\
	&&&\left.-\frac{1}{36L_{\lambda}}||\G^{k,t}_{\gamma, E}(w^k_t)||^2_2\right].
	\end{alignat}		
	Summing over $t$,
	\begin{alignat}{6} 
	&\EE\tPhi_{\lambda}(w^k_{m+1})&\leq&  
	\EE\left[\tPhi_{\lambda}(w^k_1)+\sum_{t=1}^m\left(\frac{L_{\lambda}}{6b}+\frac{13L_{\lambda}}{8t-4}\right)||\tilde{w}^k-w^k_t||^2_2\right.\nonumber\\
	&&&\left.-\sum_{t=1}^{m}\frac{13L_{\lambda}}{8t}||w^k_{t+1}-\tilde{w}^k||^2_2-\frac{1}{36L_{\lambda}}\sum_{t=1}^m||\G^{k,t}_{\gamma,
		E}(w^k_t)||^2_2\right].\nonumber
	\end{alignat}		
	As $\tilde{w}^k=w^k_1$ and $||w^k_{m+1}-\tilde{w}^k||^2_2\geq 0$,
	\begin{alignat}{6} 			
	&\EE\tPhi_{\lambda}(w^k_{m+1})&\leq&\EE\left[  
	\tPhi_{\lambda}(w^k_1)+\sum_{t=2}^m\left(\frac{L_{\lambda}}{6b}+\frac{13L_{\lambda}}{8t-4}\right)||\tilde{w}^k-w^k_t||^2_2\right.\nonumber\\
	&&&\left.-\sum_{t=1}^{m-1}\frac{13L_{\lambda}}{8t}||w^k_{t+1}-\tilde{w}^k||^2_2-\frac{1}{36L_{\lambda}}\sum_{t=1}^m||\G^{k,t}_{\gamma,
		E}(w^k_t)||^2_2\right]\nonumber\\
	&&=& 
	\EE\left[\tPhi_{\lambda}(w^k_1)+\sum_{t=1}^{m-1}\left(\frac{L_{\lambda}}{6b}+\frac{13L_{\lambda}}{8t+4}-\frac{13L_{\lambda}}{8t}\right)||w^k_{t+1}-\tilde{w}^k||^2_2\right.\nonumber\\
	&&&\left.-\frac{1}{36L_{\lambda}}\sum_{t=1}^m||\G^{k,t}_{\gamma,E}(w^k_t)||^2_2\right]\nonumber\\
	&&\leq &  
	\EE\left[\tPhi_{\lambda}(w^k_1)+\sum_{t=1}^{m-1}\left(\frac{L_{\lambda}}{6b}-\frac{L_{\lambda}}{2t^2}\right)||w^k_{t+1}-\tilde{w}^k||^2_2-\frac{1}{36L_{\lambda}}\sum_{t=1}^m||\G^{k,t}_{\gamma,
		E}(w^k_t)||^2_2\right]\nonumber\\
	&&\leq & 
	\EE\left[\tPhi_{\lambda}(w^k_1)-\frac{1}{36L_{\lambda}}\sum_{t=1}^m||\G^{k,t}_{\gamma, 		
	E}(w^k_t)||^2_2\right],\nonumber
	\end{alignat}		
	where the last inequality holds since $6b=6m^2>2(m-1)^2\geq 2t^2$ for $t=1,...,m-1$. This 	
	summation can be equivalently written as 	
	\begin{alignat}{6} 
	&&\EE\tPhi_{\lambda}(\tilde{w}^{k+1})&\leq 
	\EE\tPhi_{\lambda}(\tilde{w}^k)-\EE\left[\frac{1}{36L_{\lambda}}\sum_{t=1}^m||\G^{k,t}_{\gamma, 
	E}(w^k_t)||^2_2\right]\nonumber\\
	&&\EE\left[\frac{1}{36L_{\lambda}}\sum_{t=1}^m||\G^{k,t}_{\gamma, E}(w^k_t)||^2_2\right]&\leq 
	\EE\tPhi_{\lambda}(\tilde{w}^k)-\EE\tPhi_{\lambda}(\tilde{w}^{k+1})\nonumber\\
	&&\EE\left[\frac{1}{36L_{\lambda}}\sum_{k=1}^S\sum_{t=1}^m||\G^{k,t}_{\gamma,	
	E}(w^k_t)||^2_2\right]&\leq 
	\tPhi_{\lambda}(\tilde{w}^1)-\EE\tPhi_{\lambda}(\tilde{w}^{S+1})\nonumber\\
	&&&\leq\tPhi_{\lambda}(\tilde{w}^1)-\tPhi_{\lambda}(w^*_{\lambda})\nonumber\\
	&&\EE\left[||\G^{R,T}_{\gamma, 
		E}(w^R_T)||^2_2\right]&\leq\frac{36L_{\lambda}\left(\tPhi_{\lambda}(\tilde{w}^1)-\tPhi_{\lambda}(w^*_{\lambda})\right)}{Sm}\nonumber\\
	&&&=\tilde{\DD}\frac{L+(Sm)^{\theta}}{Sm}.\nonumber
	\end{alignat}				
\end{proof}	
\section{Proofs of Theorem 8 and Corollary 9}
\begin{customtheorem}{8}	
	Assume that $g(w)$ is Lipschitz continuous with parameter $l$ and 
	$\bar{\gamma}=\frac{1}{N^{\tau}}$ for $\tau\leq\theta$. The output $\bar{w}^R_T$ 
	of VRSPA satisfies the following inequality. 	
	\begin{alignat}{6}
	&\EE\left[\dist(0,\G_{\bar{\gamma}}(\bar{w}^R_T))\right]&\leq&\sqrt{\frac{\left(L+(Sm)^{\theta}\right)
			\left(\DD+18l^2(Sm)^{-\theta}\right)}{Sm}}+\frac{2l}{(Sm)^{\theta}}\left(2N^{\tau}+L\right),\nonumber
	\end{alignat}	
	where $\DD=36(\Phi(w^1)-\Phi(w^*))$ and $w^*$ is a global minimizer of $\Phi(\cdot)$.		
\end{customtheorem}
\begin{proof}
	The proof follows what was done to prove Theorem 5. We first verify that 
	$\bar{\gamma}=\frac{1}{N^{\tau}}\geq\frac{1}{(Sm)^{\tau}}\geq 
	\frac{1}{(Sm)^{\theta}}\geq\frac{1}{6(L+(Sm)^{\theta})}=\gamma$. From Property 3, 
	$$\dist(0,\G_{\bar{\gamma}}(\bar{w}^R_T))\leq 	
	||\G^{R,T}_{\gamma, E}(w^R_T)||_2+2l\lambda\left(\frac{2}{\bar{\gamma}}+L\right).$$	
	Taking its expectation, 	
	\begin{alignat}{6}
	&\EE\left[\dist(0,\G_{\bar{\gamma}}(\bar{w}^R_T)))\right]&\leq& 
	\EE[||\G^{R,T}_{\gamma,E}(w^R_T)||_2]+2l\lambda\left(\frac{2}{\bar{\gamma}}+L\right)\nonumber\\
	&&\leq&\sqrt{\EE\left[||\G^{R,T}_{\gamma,			
	E}(w^R_T)||^2_2\right]}+\frac{2l}{(Sm)^{\theta}}\left(2N^{\tau}+L\right)\nonumber\\
	&&\leq&\sqrt{\frac{\left(L+(Sm)^{\theta}\right)			
			\left(\DD+18l^2(Sm)^{-\theta}\right)}{Sm}}+\frac{2l}{(Sm)^{\theta}}\left(2N^{\tau}+L\right),\nonumber
	\end{alignat}	
	where the third inequality follows from Lemma 7	and using Property 4,   	
	\begin{alignat}{6}
	&&\tilde{\DD}=36(\tPhi_{\lambda}(w^1)-\tPhi_{\lambda}(w^*_{\lambda}))&\leq 36(\Phi(w^1)-\Phi(w^*)) 
	+18l^2\lambda\nonumber\\
	&&&=\DD+\frac{18l^2}{(Sm)^{\theta}}\nonumber
	\end{alignat} 			
\end{proof}
\begin{customcorollary}{9}
	Assume that $g(w)$ is Lipschitz continuous with parameter $l$. To obtain an $\epsilon$-accurate 
	solution (9) using VRSPA, the gradient call complexity is 
	$O(n^{\frac{2}{3}}\epsilon^{-3})$ and the proximal operator complexity is $O(\epsilon^{-3})$ 	
	choosing $\alpha=\theta=\frac{1}{3}$, and $\tau=0$. 		
\end{customcorollary}
\begin{proof}
	From Theorem 8,	
	\begin{alignat}{6}
	&\EE\left[\dist(0,\G_{\bar{\gamma}}(\bar{w}^R_T))\right]&\leq&\sqrt{\frac{\left(L+(Sm)^{\frac{1}{3}}\right)
			\left(\DD+18l^2(Sm)^{\frac{-1}{3}}\right)}{Sm}}+\frac{2l}{(Sm)^{\frac{1}{3}}}\left(2+L\right)\nonumber\\
	&&&=O((Sm)^{-\frac{1}{3}}).\nonumber
	\end{alignat}		
	An $\epsilon$-accurate solution will require at most  $Sm=O(\epsilon^{-3})$ iterations. Two 
	proximal operations are required each iteration, so the proximal operator complexity is 
	$O(\epsilon^{-3})$. The number of gradient calls after $Sm$ iterations is 
	$$Sn+Smb=Sm\frac{n}{\lceil n^{\frac{1}{3}}\rceil}+Sm\lceil 
	n^{\frac{1}{3}}\rceil^2=O(n^{\frac{2}{3}}\epsilon^{-3}).$$		
\end{proof}

\bibliography{PSGM}